\documentclass[10pt,a4paper]{article}
\usepackage{amsmath,amsthm,amsfonts,amssymb}
\usepackage{url}
\usepackage{gastex}

\usepackage[usenames,dvipsnames]{color}
\newenvironment{red}{\color{red}}{}

\newcommand{\ov}[1]{{\overline{#1}}}

\title{Stallings graphs for quasi-convex subgroups}

\author{
  Olga Kharlampovich\thanks{Partially supported by the PSC-CUNY award and NSF grant DMS-1201550.}\\
  \textit{Hunter College, CUNY}\\
  \texttt{okharlampovich@gmail.com}
  \and
  Alexei Miasnikov\thanks{Partially supported by NSF grants DMS-1318716 and  DMS-1201550.}\\
    \textit{Stevens Institute of Technology}\\
  \texttt{amiasnikov@gmail.com}
  \and
  Pascal Weil\thanks{Partially supported by the Programme IdEx Bordeaux - CPU (ANR-10-IDEX-03-02) and by the DeLTA project (ANR-16-CE40-0007).}\\
  \textit{LaBRI, CNRS and Universit\'e de Bordeaux}\\
  \texttt{pascal.weil@labri.fr}
}

\date{\today}

\newtheorem{theorem}{Theorem}[section]
\newtheorem{cor}[theorem]{Corollary}
\newtheorem{prop}[theorem]{Proposition}
\newtheorem{lemma}[theorem]{Lemma}

\newtheorem{fact}[theorem]{Fact}

\newtheorem{pro-example}[theorem]{Example}
\newenvironment{example}{\begin{pro-example}\rm}{\cqfd\end{pro-example}}

\newtheorem{pro-remark}[theorem]{Remark}
\newenvironment{remark}{\begin{pro-remark}\rm}{\cqfd\end{pro-remark}}

\newtheorem{pro-definition}[theorem]{Definition}
\newenvironment{definition}{\begin{pro-definition}\rm}{\cqfd\end{pro-definition}}

\def\CB{{\cal B}}

\def\CA{{\cal A}}
\def\LG{L_\geod}

\def\BP{\textsf{BP}}
\def\rBP{\textsf{rBP}}
\def\Cayley{\textsf{Cayley}}
\def\Rel{\textsf{Rel}}
\def\dist{\textsf{dist}}
\def\hauteur{\textsf{height}}

\def\Schreier{\textsf{Schreier}}
\def\WP{\textsf{WP}}
\def\calA{\mathcal{A}}

\def\calH{\mathcal{H}}
\def\calJ{\mathcal{J}}
\def\calL{\mathcal{L}}
\def\calP{\mathcal{P}}
\def\bbA{\mathbb{A}}
\def \N {\mathbb{N}}
\def\bbP{\mathbb{P}}
\def \Z {\mathbb{Z}}
\def\PSL{\mathsf{PSL}(2,\Z)}

\let\phi\varphi

\let\epsilon\varepsilon

\newcommand{\abs}[1]{\vert #1 \vert}
\newcommand{\inv}{^{-1}}

\newcommand{\geod}{\textsf{geod}}
\newcommand{\reduc}{\textsf{red}}

\newcommand{\transition}[1]{\buildrel#1\over\longrightarrow}

\def\preuve{\begin{proof}}
\def\eop{\end{proof}}
\def\cqfd{\skip10=\parfillskip\parfillskip=0pt
\enspace\hfill\symbolecqfd\par\parfillskip=\skip10\par\medskip}
\def\symbolecqfd{\rlap{$\sqcap$}$\sqcup$}
\def\proofof#1{\rm \trivlist \item[\hskip \labelsep{\it Proof of~#1.}]}
\def\eopo{\cqfd\endtrivlist}

\renewcommand{\O}{\mathcal{O}}

\begin{document}

\maketitle

\begin{abstract}
We show that one can define and effectively compute Stallings graphs for quasi-convex subgroups of automatic groups (\textit{e.g.} hyperbolic groups or right-angled Artin groups). These Stallings graphs are finite labeled graphs, which are canonically associated with the corresponding subgroups. We show that this notion of Stallings graphs allows a unified approach to many algorithmic problems: some which had already been solved like the generalized membership problem or the computation of a quasi-convexity constant (Kapovich, 1996); and others such as the computation of intersections, the conjugacy or the almost malnormality problems.

Our results extend earlier algorithmic results for the more restricted class of virtually free groups. We also extend our construction to relatively quasi-convex subgroups of relatively hyperbolic groups, under certain additional conditions.
\end{abstract}

\section{Introduction}

One can associate with any finitely generated subgroup $H$ of a free group $F$ its \emph{Stallings graph}, a uniquely defined, effectively computable, finite, directed, labeled graph with a designated base vertex, in which the reduced words labeling loops at the base vertex are exactly the elements of $H$. This remarkable combinatorial tool was introduced by Stallings \cite{1983:Stallings}, along with a simple algorithm to compute it given a tuple of generators of $H$ (subsequently improved by Touikan \cite{2006:Touikan}). Once we have Stallings graphs for the finitely generated subgroups $H, K \le F$, we can compute bases for them, compute their index, solve the membership and the conjugacy problems, find intersections of their conjugates and solve many other algorithmic problems, see \cite{2002:KapovichMyasnikov}. In short, Stallings graphs turn out to be extremely versatile tools to prove structural properties of subgroups of free groups, as well as to solve algorithmic problems.

The objective of this paper is to propose a generalization of that construction to a much more general context, which includes the quasi-convex subgroups of hyperbolic groups, and to exploit it to solve algorithmic problems.

First, let us make clear what we mean by a Stallings graph. Given a finitely generated group $G$ and a language $L$ of representatives of the elements of $G$, the \emph{Stallings graph} of a subgroup $H$ (\emph{with respect to $L$}) is the fragment $\Gamma_L(H)$ of the Schreier graph $\Schreier(G,H)$ spanned by the loops at vertex $H$, labeled by the $L$-representatives of the elements of $H$. It is a canonical object, associated with $H$ and $L$.

While this definition can always be considered, that graph may not be finite. It is the case exactly if $H$ is \emph{$L$-rational} (that is: the set of $L$-representatives of the elements of $H$ is a regular language), a notion introduced by Gersten and Short \cite{1991:GerstenShort}, who proved that it is equivalent to a geometric notion of \emph{$L$-quasi-convexity}. Such a subgroup is always finitely generated. Classical quasi-convexity corresponds to the case where $L$ is the language of geodesic representatives of the elements of $G$.

In order to prove computability results, we also assume that the group $G$ is automatic, and that the language of representatives induced by the automatic structure, is the language $L$ with respect to which we define $L$-quasi-convexity.  Recall that every hyperbolic group $G$ is geodesically automatic and that, in that case, a subgroup $H$ is quasi-convex if and only if it is undistorted, or quasi-isometrically embedded in $G$.

We observe that our definition of a Stallings graph was already considered by Gitik \cite[Definition 6]{1996:Gitik} under the name of `core', in the context of quasi-convex subgroups of hyperbolic groups. This leaves open the problem of computing the Stallings graph of a quasi-convex subgroup, when we are not given a constant of quasi-convexity for the subgroup. There exist software tools to enumerate cosets (one of the more notable is \textsf{KBMAG} \cite{2000:Holt}, see also \cite{1999:HoltHurt}), which allow exploring the Schreier graph of a subgroup: they provide a semi-algorithm that will stop only if the subgroup has finite index.

In Section~\ref{sec: effective computation}, we solve the problem of computing the Stallings graph of an $L$-quasi-convex subgroup. More precisely, we give a partial algorithm which, on input a tuple generating a subgroup $H$ of $G$, halts exactly if $H$ is $L$-quasi-convex, and in that case computes $\Gamma_L(H)$. This means in particular that if $G$ is locally quasi-convex (e.g. a surface group), then the algorithm halts on every input.

Kapovich \cite{1996:Kapovich} solved a very closely related problem: under the same hypotheses as us, he computes an automaton accepting the set of $L$-representatives of $H$ (also by a partial algorithm, halting exactly when the input tuple generates an $L$-quasi-convex subgroup). Kapovich uses it to solve the membership problem and to compute a constant of $L$-quasi-convexity for $H$.

The Stallings graph $\Gamma_L(H)$ concisely carries more information than the set of $L$-representatives of the elements of $H$, and our construction provides a unified approach to many algorithmic problems, including those solved by Kapovich. In Section~\ref{sec: algorithmic problems}, we show that if we have computed the Stallings graphs of an $L$-quasi-convex subgroup $H$, we can solve the membership problem in $H$, we can find a constant of $L$-quasi-convexity for $H$, and decide whether $H$ is finite; and if we have the Stallings graphs of $L$-quasi-convex subgroups $H$ and $K$, we can compute the intersection of these subgroups.

In Section~\ref{sec: BICI}, we identify a property of the automatic structure of $G$, a quantitative version of Hruska's and Wise's bounded packing property \cite{2009:HruskaWise}, which implies that, given the Stallings graphs of $L$-quasi-convex subgroups $H$ and $K$, one can decide whether they are conjugated, one can compute a finite family $\calJ$ of subgroups of $K$ such that every intersection of the form $K \cap H^g$ ($g\in G$) is conjugated in $K$ to an element of $\calJ$, one can compute the height of $H$ (see \cite{1998:GitikMitraRips}) and one can decide whether $H$ is almost malnormal (every intersection of $H$ with a conjugate $H^g$, $g\not\in H$, is finite). We show in Section~\ref{sec: conjugacy in hyperbolic} that the geodesic automatic structure of a hyperbolic group satisfies this bounded packing property, so the properties listed above are computable or decidable for the quasi-convex subgroups of hyperbolic groups.

In the last section of the paper, Section~\ref{sec: relatively hyperbolic}, we discuss the application of our results and methods to relatively quasi-convex subgroups of relatively hyperbolic groups. Let $G$ be a relatively hyperbolic group, relative to the peripheral structure $\calP$ and let $H$ be a relatively quasi-convex subgroup of $G$. We first show that if $H$ is peripherally finitely generated (that is: every subgroup of the form $H \cap P^g$, $g\in G$ and $P\in\calP$, is finitely generated), then $H$ is finitely generated.

Then we show that if the peripheral structure of the relatively hyperbolic group $G$ consists of geodesically bi-automatic groups and if $H$ has peripherally finite index (that is: every subgroup of the form $H^g \cap P$, for $g\in G$ and $P\in \calP$, either is finite or has finite index in $P$), then one can compute a Stallings graph and a constant of relative quasi-convexity for $H$. Like before, this is a partial algorithm which, on input a tuple of elements of $G$, halts on a set of inputs which includes the tuples generating a relatively quasi-convex subgroup with peripheral finite index. Based on this algorithm, one can decide the membership problem and the finiteness problem for $H$, and one can compute its intersection with another relatively quasi-convex subgroup. Under the same hypothesis, we show that we can compute the relative height of $H$, and decide whether $H$ is relatively malnormal. If $H, K \le G$ satisfy these same hypothesis, we can compute a finite family of representatives (up to conjugacy in $K$) of the intersections of the form $K \cap H^g$ ($g\in G$) which are so-called hyperbolic (also known as loxodromic) subgroups of $G$. The latter result can be extended to all intersections of the form $K \cap H^g$, if $G$ is toral relatively hyperbolic, leading to a decision procedure for the conjugacy problem (for relatively quasi-convex subgroups with peripherally finite index).

Comparable results were obtained earlier for significantly smaller classes of groups. McCammond and Wise solve a number of algorithmic results, including the construction of Stallings-like graphs and the solution of the membership problem, in large classes of presentations (characterized by a so-called \emph{perimeter reduction} hypothesis), many of which turn out to be locally quasi-convex \cite{2005:McCammondWise}. Schupp \cite{2003:Schupp} applies these results to large classes of Coxeter groups. In both \cite{2005:McCammondWise} and \cite{2003:Schupp}, the construction of Stallings-like graphs can be performed efficiently, in polynomial time only.

Markus-Epstein \cite{2007:Markus-Epstein} efficiently computes Stallings graphs for the finitely geneerated subgroups of amalgamated products of finite groups and uses them to solve the conjugacy and the finite index problems for these subgroups. Silva, Soler-Escriva, Ventura \cite{2016:SilvaSoler-EscrivaVentura} do the same for the finitely generated subgroups of virtually free groups. In both cases, the groups considered are locally quasi-convex.

Delgado and Ventura \cite{2013:DelgadoVentura} offered a different generalization, outside the class of hyperbolic groups, enriching the labeled graph structure to represent subgroups of direct products of free and free abelian groups: these groups do not satisfy Howson's property and Delgado and Ventura give a decision procedure to decide whether the intersection of two finitely generated subgroups is again finitely generated, and to compute it in that case.

Extending beyond automatic groups, Kapovich, Miasnikov, Weidmann \cite{2005:KapovichWeidmannMyasnikov} solve the membership problem, and more generally give an algorithmic version of the Bass-Serre theory for finitely generated subgroups of the fundamental groups of graphs of groups, under natural algebraic assumptions on the vertex groups. They do not really construct an analogue of our Stallings graphs, but rather compute the induced graph of groups for the given subgroup. 

In yet another direction, Arzhantseva proves that, generically, a finitely presented group satisfies the Howson property \cite{1998:Arzhantseva}. This result is obtained by showing that, for each integer $\ell$, there exists a generic class of presentations in which every $\ell$-generated subgroup is quasi-convex. She uses the same method to show that, generically, a subgroup of infinite index is free \cite{2000:Arzhantseva}. Kapovich and Schupp use the same idea to show that, for a generic class of 1-relator groups, the isomorphism problem is decidable \cite{2005:KapovichSchupp}.

The paper is organized as follows. Sections~\ref{sec: fp groups} and \ref{sec: Stallings graphs} introduce the necessary notions on finitely presented groups and on quasi-convexity, as well as our definition of the Stallings graph of a subgroup. Section~\ref{sec: effective computation} describes our partial algorithm to compute the Stallings graph of an $L$-quasi-convex subgroup $H$ of an automatic group, given a tuple of generators for $H$. The complexity of this algorithm is discussed in Section~\ref{sec: complexity}. Algorithmic applications are presented in Section~\ref{sec: algorithmic problems}, and the extension of our techniques to relatively hyperbolic groups is discussed in Section~\ref{sec: relatively hyperbolic}.

\section{Finitely presented groups: automatic, hyperbolic and otherwise}\label{sec: fp groups}

Let $A = \{a_1, \ldots,a_r\}$ ($r\ge 2$) be a finite alphabet (\textit{i.e.} set) and let $\tilde A$ be the alphabet $\tilde A=\{ a_1, \ldots , a_r, a_1^{-1}, \ldots, a_r^{-1} \}$, formed with the elements of $A$ and their formal inverses. We denote by $\tilde A^*$ the free monoid with basis $\tilde A$, that is, the set of all words over the alphabet $\tilde A$ with the binary operation of concatenation. By $1$ we denote the empty word in $\tilde A^*$, as well as the identity element in a monoid (in particular, a group). 

The map $^{-1}:\tilde A \to \tilde A$ defined by $a \mapsto a^{-1}, a^{-1} \mapsto a$, where $a \in A$,   extends to $\tilde A^* \to \tilde A^*$ in the usual way, so $w \mapsto w^{-1}$. A \emph{free group cancellation} is the operation of deleting, in a word of $\tilde A^*$, a factor of the form $aa\inv$ or $a\inv a$, where $a\in  A$. The \emph{free group reduction} of a word $w$ is the word $\reduc(w)$ obtained from $w$ by performing repeatedly all possible free group cancellations. The word $\reduc(w)$ is uniquely defined, i.e., it does not depend on how one performs free cancellations. A word $w \in \tilde A^*$ is called  {\em reduced} if $w = \reduc(w)$. The free group with basis $A$, written $F(A)$ can be viewed as  the set of all reduced words in $\tilde A^*$ with the standard multiplication, namely $w \cdot w' = \reduc(ww')$.

Throughout this paper, we fix an alphabet $A$, a group $G$ and an epimorphism $\mu\colon \tilde A^* \to G$ which respects inverses (a \emph{choice of generators} for $G$). We also denote by $\mu$ the restriction of $\mu$ to $F(A) \subset \tilde A^*$ and write $\ov w$ for $\mu(w)$.

The \emph{word problem} of $G$, relative to the given choice of generators, is the set $\WP_A(G) $ of words $w\in \tilde A^*$ such that $\overline w = 1$ (we also write $\WP(G)$ if $A$ is understood). The word problem is said to be \emph{solvable}  or \emph{decidable} if we can decide whether or not a given word $w\in \tilde A^*$  is in $\WP(G)$.

If $w\in \tilde A^*$ and $g\in G$, we denote by $|w|$ the length of the word $w$, and by $|g|$ the length of a shortest word in $\tilde A^*$ whose $\mu$-image is $g$. A word $w$ is called \emph{geodesic} if $|w| = |\ov w|$, and we let $\LG$ be the set of all geodesic words.

The Cayley graph of $G$ with respect to $A$ (and to $\mu$) is denoted by $\Cayley_A(G)$, or simply $\Cayley(G)$ if $A$ is understood. Its vertex set is $G$, and it has an $a$-labeled edge from $\ov w$ to $\overline{wa}$ for every $w\in \tilde A^*$ and $a\in A$.

The choice of $\mu$ determines a metric $d$ on $G$, and on $\Cayley_A(G)$, defined by $d(g,h) = \abs{g\inv h}$.

\subsection{Rational structures for finitely generated groups}\label{sec: rational structure}

A \emph{rational language} $L$ over $\tilde A$ is a subset $L  \subseteq \tilde A^*$ which is accepted by a finite state automaton. The pair $(G,L)$ is a \emph{rational structure of $G$ relative to $A$} if $L \subseteq \tilde A^*$ is a rational language consisting only of reduced words, and such that $\mu(L) = G$. If $g \in G$, an element of $L$ in $\mu\inv(g)$ is called an \emph{$L$-representative} of $g$.

Given a rational structure $(G,L)$ on $G$, a subset $S \subseteq G$ is called $L$-\emph{rational} if the set $L \cap \mu^{-1}(S)$ is a rational subset of $\tilde A^*$.

The study of rational structures of groups was initiated by Gilman \cite{1979:Gilman,1987:Gilman} and was one of the ingredients that gave rise to the theory of automatic groups \cite{1992:EpsteinCannonHolt}.

\subsection{Automatic groups}\label{sec: automatic groups}

An \emph{automatic structure} for the $A$-generated group $G$ consists of the following collection of finite state automata:

- an automaton $\calA_{acc}$ (the \emph{word acceptor}) accepting a set $L$ of representatives of the elements of $G$ (that is: $\mu(L) = G$);

- for each letter $a\in \tilde A \cup \{1\}$, an automaton $\calA_a$ (the \emph{$a$-multiplyer}) accepting the pairs $(u,v)$ such that $u,v\in L$ and $\overline{ua} = \overline v$.

When we say that an automaton accepts a set $K \subseteq \tilde A^*\times \tilde A^*$, we actually mean the following: the automaton operates on alphabet $(\tilde A\cup\{\square\})^2$ (where $\square$ is a new symbol) and it accept the words on that alphabet of the form $(u\square^n,v\square^m)$, where $u,v\in \tilde A^*$, $|u|+n = |v|+m$, $\min(n,m) = 0$ and $(u,v) \in K$.

If the set of representatives is the set $\LG$ of geodesics, we talk of a \emph{geodesically automatic} structure.

If $G$ admits an automatic structure, then $G$ is called an \emph{automatic group}. We refer to the book \cite{1992:EpsteinCannonHolt} for the standard facts about automatic groups. Hyperbolic groups and right-angled Artin groups are geodesically automatic groups.

It is no restriction to request that the set $L$ of representatives of the elements of $G$ consists only of reduced words, and we will always assume that this is the case. As a result, an automatic structure for $G$ yields a rational structure $(G,L)$ in the sense of Section~\ref{sec: rational structure}.

\subsection{Group presentations and the Dehn algorithm}

Let $G = \langle A \mid R\rangle$ be a finite  presentation of the group $G$, with $R$ a set of cyclically reduced words in $\tilde A^*$. We assume  that this presentation is symmetric, i.e., $R^{-1} = R$, and closed under cyclic permutations, i.e., if $r \in R$ and $r = r_1r_2$ then $r_2r_1 \in R$. The \emph{Dehn rewriting system associated to $R$} is a finite string rewrite system $D_R$, defined as follows:
$$D_R = \{ r_1 \to r_2^{-1} \mid r_1r_2 \in R , |r_1| > |r_2|\} \cup \{aa^{-1} \to 1 \mid a \in \tilde A\}.$$
A word $w\in \tilde A^*$ is called \emph{Dehn-reduced} if no rule of $D_R$ can be applied to it, that is, if $w$ contains no occurrence of a left side of a rule of $D_R$. The rewriting system $D_R$ is length-reducing by construction, and hence terminating. This means that, given a word $w\in \tilde A^*$, one can apply a finite (possibly empty) sequence of rules of $D_R$ to yield a Dehn-reduced word (a Dehn reduced form of $w$). Obviously, this  can be done in polynomial time in the length $|w|$ of $w$. Note that, in general, applying different sequences of rules of $D_R$ to the word $w$ may lead to different Dehn-reduced forms of $w$  (that is: the rewriting system $D_R$ may not be \emph{confluent}). Observe that $w$ and all its Dehn-reduced forms represent the same element in $G$.

We denote by $L_R \subseteq F(A)$ the language of Dehn-reduced words relative to $D_R$. Then
$L_R$ is the set of words in $\tilde A^*$ avoiding a finite set of factors: the length two words of the form $aa\inv$ and $a\inv a$ with $a\in A$ and the words $r_1$ which are \emph{long} prefixes of a relator $r\in R$, that is, which satisfy $|r_1| > \frac12|r|$. It is a standard result of automata theory that such a language is rational. Furthermore, one can effectively construct  an automaton accepting $L_R$. Moreover $(G,L_R)$ is a rational structure for $G$.

The \emph{Dehn algorithm} --- an  attempt to solve the word problem of $G$ --- works as follows: on a given input word $w\in \tilde A^*$, the algorithm computes a Dehn-reduced form of $w$, say $\rho(w)$, by applying the rules from $D_R$ in some fixed order. The algorithm then returns the answer \textsf{Yes} if $\rho(w)$ is the empty word, and the answer \textsf{No} otherwise.

Observe that this algorithm is always correct when it returns a positive answer on an input word $w$, since in this case, indeed,  $w\in \WP(G)$. More generally, it is always the case that $\overline w = \overline{\rho(w)}$. However there are groups and presentations for which the Dehn algorithm makes mistakes on negative answers (that is: it returns the answer \textsf{No} even though $\ov w = 1$).

We say that the presentation $G = \langle A \mid R \rangle$ is a \emph{Dehn presentation} if the Dehn algorithm correctly solves the word problem.

\begin{remark}
A geodesic word is always Dehn reduced, so $\LG \subseteq L_R$. The converse is not true in general and some Dehn reduced words may not be geodesic, even if the presentation is Dehn. On this, see \cite{2000:Arzhantseva-dehn}.
\end{remark}

\subsection{Hyperbolic groups}
\label{se:intro-hyp}

Let $\delta \ge 0$. The group $G = \langle A \mid R \rangle$ is called \emph{$\delta$-hyperbolic} if, for all $g,h,k\in G$ and for all geodesic paths $p(g,h)$, $p(g,k)$ and $p(h,k)$ in $\Cayley(G)$ between these points, every vertex of $p(g,h)$ sits within distance at most $\delta$ from the union of $p(g,k)$ and $p(h,k)$. The group $G$ is called \emph{hyperbolic} if it is $\delta$-hyperbolic for some $\delta$. This geometric property of $\Cayley(G)$ has far-reaching consequences on the  algorithmic properties of $G$. We refer to \cite{1987:Gromov,1992:EpsteinCannonHolt,1999:BridsonHaefliger} for further details, and we report here only the facts that will be used directly in this paper:

-  a finitely generated group  is hyperbolic if and only if it admits a Dehn presentation \cite{1987:Gromov,1989:Lysenok};

- if $G$ is hyperbolic, then a hyperbolicity constant and a Dehn presentation  can be  computed \cite{1996:Papasoglu,2001:EpsteinHolt}.

-  a hyperbolic group $G$ admits a shortlex automatic structure, and such an automatic structure can be  computed  \cite{1992:EpsteinCannonHolt}.

\begin{remark}
Finding a hyperbolicity constant is one of the most fundamental algorithmic tasks for hyperbolic groups. There are partial algorithms (see \cite{1996:Papasoglu,2001:EpsteinHolt}) that start on a given  finite presentation and  then  stop if and only if the group defined by the presentation is hyperbolic (and in this case the algorithm spells out a hyperbolicity constant of the Cayley graph of the group). Notice, that the problem to decide whether a finite presentation defines a hyperbolic group is undecidable (since hyperbolicity of a finite presentation is a Markov property), so the time function of any partial algorithm as above cannot be bounded by a computable function. This implies that in the worst case all  algorithms that require computation of hyperbolicity constants of  hyperbolic groups are very inefficient.  In this paper, whenever a hyperbolic group $G$ is given, we consider that a hyperbolicity constant, a Dehn presentation and a shortlex automatic structure for $G$ are already computed. 
\end{remark}

\section{Stallings graphs}\label{sec: Stallings graphs}

Let $G$ be a group with a rational structure $(G,L)$. In this section, we introduce a notion of a Stallings graph associated with a subgroup $H$ of a group $G$ relative to $L$. This graph may be infinite in general. We show that it is finite if and only if $H$ is $L$-quasi-convex in $G$ (see Section~\ref{sec: quasi-convex subgroups}).  

\subsection{Reduced rooted graphs}

An \emph{$A$-graph} is a directed graph, whose edges are labeled by the letters of $A$. More formally, it is a pair of the form $\Gamma = (V,E)$ where $E \subseteq V \times A \times V$: $V$ is called the set of \emph{vertices}, and $E$ the set of \emph{edges}. If $e = (v,a,w) \in E$, then $a$ is the \emph{label} of $e$.

We say that $\Gamma$ is \emph{folded} if distinct edges starting (resp. ending) at the same vertex have distinct labels. If $v \in V$, we say that $\Gamma$ is \emph{$v$-trim} if every vertex except perhaps $v$ is incident to at least two edges. The pair $(\Gamma,v)$ is called a \emph{reduced rooted graph} if $\Gamma$ is connected (that is: the underlying undirected graph is connected), folded and $v$-trim. 

The pair $(\Gamma,v)$ can be viewed as an automaton, with the following convention: when an $a$-labeled edge is read against its orientation, we are actually reading the letter $a\inv$. We observe that if a word $w\in \tilde A^*$ can be read from vertex $p$ to vertex $q$ (labels a path from $p$ to $q$), then so can every word obtained from $w$ by a sequence of free group cancellations, including $\reduc(w)$ itself.

A reduced rooted (finite) $A$-graph $(\Gamma,v)$ yields a (finitely generated) subgroup of $F(A)$, namely the set of reduced words which label a loop at the vertex $v$, written $\calH(\Gamma,v)$. And conversely, every such subgroup $H \le F(A)$ yields a unique reduced rooted $A$-graph, called its \emph{Stallings graph} and written $(\Gamma(H),v)$.  Notice, that there is an algorithm to compute $\Gamma(H)$ from a set of generators of $H$ in almost linear time \cite{2006:Touikan}. We refer the reader to \cite{2002:KapovichMyasnikov} for further details.

\subsection{The Stallings graph of a subgroup of $G$}\label{sec: Stallings}

Let $H \le G$ be a subgroup of a group $G$. The \emph{Schreier graph} of $H$ in $G$, written $\Schreier_A(G,H)$ (or $\Schreier(G,H)$ if $A$ is understood), is the $A$-graph whose vertex is the set of all cosets $Hg$ ($g\in G$)  and whose edges are all the triples of the type  $(Hg,a,Hg\bar a)$ ($g\in G$, $a\in A$). This graph is clearly folded, and the words labeling loops at the vertex $H$ form exactly the set $\mu\inv(H)$. As before we assume that $(G,L)$ is  a fixed rational structure for $G$.

\begin{definition}
Let $(G,L)$ be a rational structure for the group $G$ and let $H$ be a subgroup of $G$. The \emph{Stallings graph of $H$ with respect to $L$} is the fragment of $\Schreier(G,H)$ rooted at the vertex $H$, and spanned by the loops at $H$ labeled by the $L$-representatives of the elements of $H$, that is, by the elements of $L\cap \mu\inv(H)$. This graph is denoted by $\Gamma_L(H)$ or $(\Gamma_L(H),H)$ to emphasize that $H$ is the root of the graph.
\end{definition}

It is easily verified that $(\Gamma_L(H),H)$ is a reduced rooted $A$-graph. Note that this definition coincides with the notion of \emph{core} introduced by Gitik \cite[Definition 6]{1996:Gitik}.

\begin{example}\label{ex: Stallings graph in PSL}
Let $G$ be the modular group, $G = \langle a, b\mid a^2, b^3\rangle$. Then $\LG$ is the set of reduced words over the 4-letter alphabet $\{a,a\inv,b,b\inv\}$, with no occurrence of two successive $a$, $a\inv$, $b$ or $b\inv$.  Figure~\ref{fig: Stallings graph in PSL} shows the Stallings graphs of the subgroups generated by $\{abab\inv, babab\}$, in $F(a,b)$ and in $\PSL$ (see Markus-Epstein \cite{2007:Markus-Epstein} for a justification in the modular case).
\begin{figure}[htbp]
\centering
\begin{picture}(88,32)(5,-32)
\gasset{Nw=4,Nh=4}

\node(n0)(27.0,-7.0){$H$}

\node(n1)(44.0,-7.0){}

\node(n2)(44.0,-29.0){}

\node(n3)(27,-29.0){}

\node(n4)(10,-29.0){}

\node(n5)(0.0,-18.0){}

\node(n6)(10.0,-7.0){}

\drawedge(n0,n1){$a$}

\drawedge(n0,n3){$b$}

\drawedge(n1,n2){$b$}

\drawedge(n2,n3){$a$}

\drawedge(n3,n4){$a$}

\drawedge(n4,n5){$b$}

\drawedge(n5,n6){$a$}

\drawedge(n6,n0){$b$}
\node(n0)(60.0,-4.0){$H$}

\node(n1)(66.0,-18.0){}

\node(n2)(60.0,-32.0){}

\node(n3)(94,-4.0){}

\node(n4)(90,-18.0){}

\node(n5)(94.0,-32.0){}

\drawedge(n0,n1){$b$}

\drawedge(n1,n2){$b$}

\drawedge(n2,n0){$b$}

\drawedge[ELside=r](n3,n4){$b$}

\drawedge[ELside=r](n4,n5){$b$}

\drawedge[ELside=r](n5,n3){$b$}

\drawedge[curvedepth=2.0](n0,n3){$a$}

\drawedge[curvedepth=2.0](n3,n0){$a$}

\drawedge[curvedepth=2.0](n1,n4){$a$}

\drawedge[curvedepth=2.0](n4,n1){$a$}

\drawedge[curvedepth=2.0](n2,n5){$a$}

\drawedge[curvedepth=2.0](n5,n2){$a$}
\end{picture}
\caption{\small Stallings graphs of subgroups of the free group and of the modular group.}\label{fig: Stallings graph in PSL}
\end{figure}
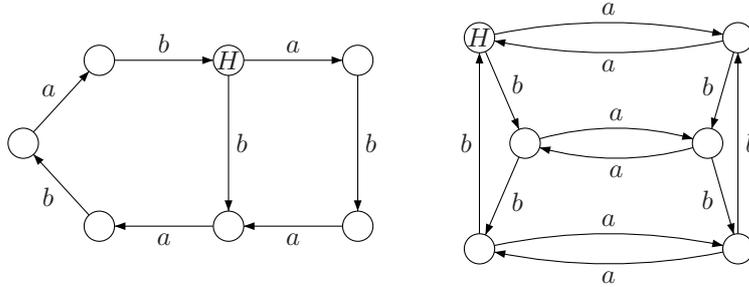
\end{example}

We will also use the following notion.

\begin{definition}
A reduced rooted $A$-graph $(\Gamma,v)$ is a \emph{Stallings-like} graph \emph{for $H$ with respect to $L$} if every loop at $v$ in $\Gamma$ is labeled by a word in $\mu\inv(H)$, and every $L$-representative of an element of $H$ labels a loop of $\Gamma$ at $v$.
\end{definition}

It follows directly from these two definitions that the Stallings graph of $H$ with respect to $L$ is Stallings-like for $H$. It is in fact minimal among the Stallings-like graphs in the following sense.

\begin{prop}\label{prop: Stallings minimal}
Let $(G,L)$ be a rational structure for $G$ and $H \le G$ a subgroup of $G$.    If $(\Gamma,v)$ is a  Stallings-like graph for $H$ with respect to $L$ then there exists a graph morphism from $(\Gamma,v)$ into $\Schreier(G,H)$, whose range contains $(\Gamma_L(H),H)$.
If in addition $(\Gamma,v)$ is a subgraph of $(\Schreier(G,H),H)$, then this graph morphism is injective.

In particular, $\Gamma_L(H)$ is the least reduced graph that is Stallings-like for $H$ with respect to $L$ (up to an isomorphism of labeled graphs), and it is the unique minimal Stallings-like subgraph of $\Schreier(G,H)$.

\end{prop}

\preuve
For each vertex $p$ of $\Gamma$, we let $\phi(p) = H\bar u$, where $u$ is a word labeling a path in $\Gamma$ from $v$ to $p$. It is easily verified that $\phi$ defines a morphism from $(\Gamma,v)$ into $\Schreier(G,H)$ (since every loop at $v$ in $H$ is labeled by an element of $\mu\inv(H)$). Note that if a morphism exists between two reduced rooted $A$-graphs, then it is unique.

Moreover, the image of $\Gamma$ contains $\Gamma_L(H)$ since every $L$-representative of an element of $H$ labels a loop in $\Gamma$ at $v$, and $\Gamma_L(H)$ is spanned by the loops labeled by these words.
\eop

\subsection{Quasi-convex subgroups}\label{sec: quasi-convex subgroups}

A subset $S \subseteq G$ is called $L$-\emph{quasi-convex}  if there exists an integer $k$ such that, for every word $w \in L \cap \mu^{-1}(S)$, the path in $\Cayley(G)$ which starts at 1 and is labeled by the word $w$, lies in the $k$-neighborhood of $S$, \emph{i.e.}, it stays always within distance at most $k$ from $S$. We then say that $k$ is a \emph{constant of $L$-quasi-convexity} for $S$. We now concentrate on the case of subgroups.

Given a subgroup $H \le G$ and an integer $k$, we denote by $\Schreier_k(G,H)$ the fragment of $\Schreier(G,H)$ consisting of the vertices at distance at most $k$ from vertex $H$, and all the edges between them. Then $(\Schreier_k(G,H),H)$ is a finite folded connected rooted $A$-graph, and we let $W_k(H)$ be the subgroup of $F(A)$, of which it is the Stallings graph (that is: $W_k(H) = \calH(\Schreier_k(G,H),H)$). Note that $\mu(W_k(H)) \subseteq H$ for every $k$.

\begin{prop}\label{prop: Stallings graphs of quasi-convex}\label{prop: Stallings minimal second}
Let $(G,L)$ be a rational structure for $G$.  Then a subgroup $H \leq G$ is $L$-quasi-convex if and only if it has a finite Stallings-like graph with respect to $L$, if and only if $\Gamma_L(H)$ is finite, if and only if $L\cap\mu\inv(H)$ is contained in $W_k(H)$ for some $k\ge 1$.

If $(\Gamma, v)$ is a rooted reduced graph which is Stallings-like for $H$ with respect to $L$, and every vertex of $\Gamma$ is at distance at most $k$ from the base vertex $v$, then $k$ is a constant of $L$-quasi-convexity for $H$.

The least constant of $L$-quasi-convexity of $H$ is the maximum distance from a vertex to the base vertex in the Stallings graph $\Gamma_L(H)$ or, equivalently, the least $k$ such that $(\Gamma_L(H),H)$ is contained in $(\Schreier_k(G,H),H)$.

If $H$ is $L$-quasi-convex with constant $k$, then $\Gamma_L(H)$ has diameter at most $2k$. If in addition $|A| = r$, then $\Gamma_L(H)$ has at most $1 + \frac r{r-1} \big((2r-1)^k-1\big)$ vertices.
\end{prop}

\preuve
The first statements follow directly from the definition of quasi-convexity. For the last statement, we note that every vertex of $\Schreier_k(G,H)$ is at distance at most $k$ from the base vertex $H$. The upper bound on the number of vertices relies on the following observation: every vertex of $\Schreier_k(G,H)$ is at the extremity of a path starting at $H$, labeled by a reduced word in $F(A)$ of length at most $k$. The number of such words is 1 at length 0 and $2r(2r-1)^{i-1}$ at length $i\ge 1$. Therefore there are at most 
$$1 + \sum_{i=1}^k 2r(2r-1)^{i-1} = 1 + 2r \frac{(2r-1)^k-1}{2(r-1)} = 1 + \frac r{r-1} \big((2r-1)^k-1\big)$$
 vertices.
\eop

The following result is well-known (see Gersten and Short \cite{1991:GerstenShort}).

\begin{theorem} \label{th:SS}
Let $(G,L)$ be a rational structure for $G$.  Then a subgroup $H \leq G$ is $L$-rational if and only if it is $L$-quasi-convex.
\end{theorem}

If $L = \LG$, the set of all geodesics, we talk of \emph{quasi-convexity} (instead of $L$-quasi-convexity). We note the following relations between $L$-quasi-convexity and quasi-convexity.

\begin{fact}\label{fact: sufficient condition for quasi-convexity equivalence}
If $\LG \subseteq L$ and $H\le G$ is $L$-quasi-convex with constant $k$, then $H$ is also quasi-convex, with the same constant.

Suppose that, for some $\ell\ge 1$, every element $w\in L$ labels a path from 1 in $\Cayley(G)$ which remains in the $\ell$-neigborhood of any geodesic from 1 to $\ov w$. If $H$ is quasi-convex with constant $k$, then $H$ is $L$-quasi-convex with constant $k+\ell$.
\end{fact}

\section{Effective computation}\label{sec: effective computation}

Let $G$ be an $A$-generated group and let $(G,L)$ be a rational structure for $G$. If $h_1, \ldots, h_s$ are reduced words in $F(A)$ such that $H = \langle \bar h_1,\ldots, \bar h_r\rangle$, we say for short that \emph{$h_1, \ldots, h_s$ generate the subgroup $H$ of $G$}. Proposition~\ref{prop: Stallings graphs of quasi-convex} established that if $H$ is $L$-quasi-convex if and only if the Stallings graph $\Gamma_L(H)$ is finite. We show the following result.

\begin{theorem}\label{thm: effective computation}
There exists a partial algorithm which, given an automatic structure for a group $G$ with language of representatives $L$, and given finitely many reduced words generating a subgroup $H$ of $G$, stops and computes the Stallings graph of $H$ with respect to $L$ if $H$ is $L$-quasi-convex, and never stops if $H$ is not $L$-quasi-convex.
\end{theorem}

%
%

The proof of Theorem~\ref{thm: effective computation}, given in Sections~\ref{sec: completion process}, \ref{sec: computing a Stallings-like graph} and \ref{sec: computing}, explicitly exhibits the partial algorithm announced in the statement.
We point out that the theorem applies to the case of the quasi-convex subgroups of hyperbolic groups or right-angled Artin groups, since such groups are geodesically automatic (see Section \ref{se:intro-hyp}). 

Before we give our construction, we record the following corollary (see also Section~\ref{sec: complexity}), a result already proved by Kapovich \cite{1996:Kapovich}. 

\begin{cor}\label{cor: compute constant}
There exists a partial algorithm which, given an automatic structure for a group $G$ with language of representatives $L$, and given finitely many reduced words generating a subgroup $H$ of $G$, stops and computes the least constant of $L$-quasi-convexity of $H$ with respect to $L$ if $H$ is $L$-quasi-convex, and never stops if $H$ is not $L$-quasi-convex.
\end{cor}

\preuve
It is a direct consequence of the definition of the Stallings graph $\Gamma_L(H)$ that the least $L$-quasi-convexity constant of $H$ is the maximal distance from the base vertex $H$ to a vertex of $\Gamma_L(H)$ (see also Proposition~\ref{prop: Stallings graphs of quasi-convex}). It is therefore computable once we have computed $\Gamma_L(H)$.
\eop

\begin{remark}\label{rk: least constant of quasi-convexity}
We note that the parameter used in the proof of Corollary~\ref{cor: compute constant}, namely the maximal distance from the root to a vertex, computed on any reduced rooted graph which is Stallings-like for $H$ with respect to $L$, is also an $L$-quasi-convexity constant for $H$, although maybe not minimal.

In \cite{1996:Kapovich}, Kapovich actually computes \emph{some} $L$-quasi-convexity constant of the subgroup $H$, and not the least one. However, once such a constant is known, one can find the least constant by some brute force exploration. Computing the Stallings graph of $H$ as we do in Theorem~\ref{thm: effective computation}, reveals the least constant much more directly.
\end{remark}

Let us now turn to the proof of Theorem~\ref{thm: effective computation}. The computation of the Stallings graph of an $L$-quasi-convex subgroup is in two steps. The first one consists in a completion process, described in Section~\ref{sec: completion process}, which does not terminate in general, yet eventually produces a Stallings-like graph for $H$ with respect to $L$, if $H$ is $L$-quasi-convex. We then show, in Section~\ref{sec: computing a Stallings-like graph}, that the automatic structure on $G$ allows us to determine when a Stallings-like graph has been produced, and hence when to stop the first phase of the algorithm (this is the only point where we use the automatic structure of $G$). This graph in turns yields a solution to the membership problem for $H$ (Corollary \ref{cor: membership problem for quasi-convex}). The last step uses this solution of the membership problem to effectively construct the Stallings graph $\Gamma_L(H)$ (Section~\ref{sec: computing}).

The complexity of this algorithm is discussed in Section~\ref{sec: complexity}.

\subsection{A completion process}\label{sec: completion process}

Let $G = \langle A \mid R \rangle$ be a finitely presented  group equipped with an automatic structure, and let $(G,L)$ be the corresponding rational structure for $G$. We assume that the set of relators $R$ is closed under cyclic permutation of its elements and under taking inverses.

Let $h_1,\ldots, h_s$ be reduced words in $F(A)$ and let $H$ be the subgroup of $G$ generated by $\ov h_1,\ldots,\ov h_s$. We make no other assumption on $H$ at this stage. The completion process we now describe does not halt in general. Corollary~\ref{cor: eventually a weak Stallings graph} below shows that it eventually produces Stallings-like graphs for $H$ if $H$ is quasi-convex. Effectively computing such a graph will be done in Section~\ref{sec: computing a Stallings-like graph}.

We define the following sequence of reduced rooted graphs.

\begin{itemize}
\item $(\Gamma_0,1)$ is  the Stallings graph of the subgroup of $F(A)$ generated by the $h_i$.
\item Suppose the graph $(\Gamma_i,1)$ is constructed. Then $\Gamma_{i+1}$ is constructed by 
\begin{itemize}
\item[(a)] adding to $\Gamma_i$, at every vertex, every relator from $R$ and every word $aa\inv$ ($a\in \tilde A$), viewed as a labeled circle; and
\item[(b)] folding the resulting labeled graph, that is, iteratively identifying pairs of edges with the same label and with equal initial (resp. terminal) vertices.
\end{itemize}
\end{itemize}

\begin{remark}
Note that we do not need to bother adding the loops labeled $aa\inv$ when letter $a$ occurs in $R$.
\end{remark}

By construction, every $(\Gamma_i,1)$ is a finite reduced rooted graph. For each $i\ge 0$, we let $\calL_i$ be the set of words (possibly not reduced) which label a loop at 1 in $\Gamma_i$, and $H_i = \reduc(\calL_i) = \calH(\Gamma_i,1) \le F(A)$.

\begin{lemma}
\label{le: increasing sequence}
For each $i\ge 0$, $\calL_i \le \calL_{i+1}$, $H_i \le H_{i+1}$ and $\mu(H_i) = H$. Moreover the subgroups $H_i \le F$ are finitely generated and $H_0 = \langle h_1, \ldots, h_s\rangle$.
\end{lemma}

\begin{proof}
It is  immediate that every word which labels a loop at 1 in $\Gamma_i$ also does so in $\Gamma_{i+1}$, so  $\calL_i \le \calL_{i+1}$ and $H_i \le H_{i+1}$. Each $H_i$ is finitely generated since $\Gamma_i$ is finite.
Notice that $H_0 = \langle h_1, \ldots, h_s\rangle$ by construction. Therefore $\mu(H_0) = H$ and hence $H \le \mu(H_i)$ for each $i$.

We prove by induction that $H = \mu(H_i)$ for each $i$.  It is already established if $i = 0$ and we assume that it holds for some $i\ge 0$. Consider a reduced word $w$ labeling a loop at $1$ in $\Gamma_{i+1}$. By construction of $\Gamma_{i+1}$, $w$ is the free  reduction of a (possibly non-reduced) word $w'$, labeling a loop at 1 in the graph $\Gamma'_i$ obtained from $\Gamma_i$ by gluing every relator from $R$, viewed as a labeled circle, to every vertex. Such a path can be analyzed as a loop in $\Gamma_i$ at vertex 1 with label, say $w''$, with relators inserted as loops at various points. In particular, $\overline{w''} = \overline{w'}$ since $\overline r = 1$ for each $r\in R$, so $\overline{w''} \in H$ by the induction hypothesis, and hence $\overline w = \overline{w''} \in H$.
\end{proof}

Let $\widetilde D_R$ be the follows rewriting system
$$\widetilde D_R = \{ r_1 \to r_2 \mid r_1r_2\inv \in R\} \cup \{aa^{-1} \to 1 \mid a \in \tilde A\} \cup \{1 \to aa^{-1} \mid a \in \tilde A\}.$$
Obviously, two words $u,v\in \tilde A^*$ determine the same element in $G$  if and only if one can pass from $u$ to $v$ by a sequence of rewrites  using rules of the system $\widetilde D_R$. In particular, the non-reduced word problem of $G$, namely the set $\widetilde{\WP}(G)$ of all words in $\tilde A^*$ defining  the trivial element in $G$, consists of all  words that can be deduced from the empty word using the rules of $\widetilde D_R$.

We show that the sequence $(\calL_i)_i$ gradually includes all the words from $\widetilde\WP(G)$.

\begin{prop}
\label{le: incorporate more rules}
Let $i\ge 0$. If a word $w\in \widetilde\WP(G)$ is obtained from 1 by the application of a sequence of at most $i$ rules from $\widetilde D_R$, then $w\in \calL_i$.
\end{prop}

\begin{proof}
The proof is by induction on $i$, and the result is trivial for $i = 0$. Let us assume that the result holds for some $i\ge 0$ and suppose that $w$ is obtained from 1 by a sequence of $i+1$  applications of rules from $\widetilde D_R$. Let $w_1$ be the word obtained after the application of the first $i$ rules of this sequence. Then, by induction, $w_1 \in \calL_i$, that is, $w_1$ labels a loop at 1 in $\Gamma_i$. Moreover, one can go from $w_1$ to $w$ by the application of a single rule of $\widetilde D_R$.

If $w$ is obtained from $w_1$ by deleting a factor $aa\inv$ ($a\in \widetilde A$), then $w$ also labels a loop at 1 in $\Gamma_i$ (since $\Gamma_i$ is reduced), that is: $w\in \calL_i$, and hence $w\in \calL_{i+1}$.

If $w$ is obtained from $w_1$ by inserting a factor $aa\inv$ ($a\in \widetilde A$), then there exist $x,y\in F(A)$ such that $w_1 = xy$ and $w = xaa\inv y$. In particular, $\Gamma_i$ has an $x$-labeled path from 1 to a vertex $p$ and a $y$-labeled path from $p$ to 1. The construction of $\Gamma_{i+1}$ from $\Gamma_i$ includes the addition of an $a$-labeled edge starting at vertex $p$, so after folding, $w$ labels a loop at 1 in $\Gamma_{i+1}$.

Finally, if $w$ is obtained from $w_1$ by substituting $r_2$ for $r_1$, for some $r_1,r_2\in F(A)$ such that $r_1r_2\inv\in R$, then there exist $x,y\in F(A)$ such that $w_1 = xr_1y$ and $w = xr_2y$. In particular, $\Gamma_i$ has an $x$-labeled path from 1 to a vertex $p$, an $r_1$-labeled path from $p$ to a vertex $q$ and a $y$-labeled path from $q$ to 1. The construction of $\Gamma_{i+1}$ from $\Gamma_i$ includes the addition of an $r_1r_2\inv$-labeled loop at vertex $p$. After folding, there results an $x$-labeled path from 1 to a vertex $p'$, and an $r_2y$-labeled path from $p'$ to 1. Therefore $w\in \calL_{i+1}$.
\end{proof}

As a result of Proposition~\ref{le: incorporate more rules}, the sequence $H_i$ gradually includes all words in $\mu\inv(H)$, and all finitely generated subgroups contained in $\mu\inv(H)$.

\begin{cor}\label{cor: encompass all words in mu inv H}
Let $w\in \tilde A^*$. If $\mu(\reduc(w)) \in H$, then $w \in \calL_i$ and $\reduc(w) \in H_i$ for some $i\ge 0$.
\end{cor}

\begin{proof}
If $\mu(\reduc(w)) \in H$, then $\mu(\reduc(w)) = \mu(w')$ for some $w'\in H_0$. Then $w'w\inv \in \widetilde\WP(G)$, and hence $w'w\inv\in \calL_i$ for some $i$. Since $w'\in H_0 \le H_i$ and $w'$ is a reduced word, we have $w'\in \calL_i$, which yields $w\in\calL_i$.
\end{proof}

\begin{cor}\label{cor: H_i encompasses every reduced fg subgroup}
Let $K\le F(A)$ be a finitely generated subgroup, such that $\mu(K) \le H$. Then there exists an integer $i$ such that $K \le H_i$.
\end{cor}

\begin{proof}
Let $k_1,\ldots, k_p$ be a basis of $K$. By Corollary~\ref{cor: encompass all words in mu inv H}, for each $1\le j \le p$, there exists an integer $i_j$ such that $k_j \in H_{i_j}$. Let $i = \max(i_j \mid 1 \le j \le p)$. Then each $k_j \in H_i$, and hence $K \le H_i$.
\end{proof}

Applying Corollary~\ref{cor: H_i encompasses every reduced fg subgroup} to the subgroups defined by the sequence of graphs $\Schreier_k(G,H)$ ($k\ge 1$), we see that the $H_i$ gradually incorporate reduced words whose $\mu$-image is in $H$, yet label paths in $\Cayley_A(G)$ which go further and further from $H$.

\begin{cor}\label{W_k in some Gamma_i}
For each integer $k\ge 1$, there exists an integer $i$ such that $W_k(H) \le H_i$.
\end{cor}

If the subgroup $H$ is $L$-quasi-convex with constant $k$, then $W_k(H)$ contains all the $L$-representatives of every element of $H$. Thus, Corollary~\ref{W_k in some Gamma_i} implies the following statement.

\begin{cor}\label{cor: eventually a weak Stallings graph}
If $H$ is $L$-quasi-convex, then $\Gamma_i$ is a Stallings-like graph for $H$ with respect to $L$ for every large enough $i$, namely, for every $i$ at least equal to an $L$-quasi-convexity constant.
\end{cor}

Note that this statement does not solve the question of computing a Stallings-like graph.

\subsection{Computing Stallings-like graphs}\label{sec: computing a Stallings-like graph}

Let $G, L, H, h_1, \ldots, h_s$ be as in Section~\ref{sec: completion process}. We now assume that $H$ is $L$-quasi-convex -- but we are not given an $L$-quasi-convexity constant. We first show that one can decide whether a given reduced rooted graph is Stallings-like for $H$ with respect to $L$. This requires the following technical lemma.

\begin{lemma}\label{lemma: compute multiplyer}
There is an algorithm which, given an automatic structure for a group $G$, with set of representatives $L$, and given $h\in F(A)$ and a rational language $K$ contained in $L$, constructs an automaton accepting the set $M_h(K)$ of all the $L$-representatives of the elements of $\mu(Kh)$.
\end{lemma}

\begin{proof}
For each $a\in \tilde A$, let $\CA_a$ be the $a$-multiplyer automaton given by the automatic structure of $G$, which accepts the language of all words on alphabet $(\tilde A\cup\{\square\})^2$ of the form $(u\square^k,u\square^\ell)$, where $|u|+k = |v|+\ell$, $\min(k,\ell) = 0$, $u,v\in L$ and $\ov v = \ov{ua}$ (see Section~\ref{sec: automatic groups}).

We first extend these automata to allow arbitrary padding by $\square$ symbols at the end of accepted sequences: we add to each $\CA_a$ a state $s_a$, $(\square,\square)$-labeled edges from the accepting states to $s_a$, and a $(\square,\square)$-labeled loop at $s_a$. We let $Q_a$ be the state set of the resulting automaton.

Now let $\CA$ be a deterministic and complete automaton accepting $K$. Again, we extend $\CA$ by adding a state $s$, $\square$-labeled edges from the accepting states to $s$ and a $\square$-labeled loop at $s$. Let $Q$ be the resulting state set.

Now let $h = a_1\cdots a_n\in \tilde A^*$. We construct an automaton $\CB$ with state set $Q\times Q_{a_1} \times \cdots \times Q_{a_n}$ as follows. For each $\vec b = (b_0,b_1,\ldots, b_n) \in (\tilde A\cup\{\square\})^{n+1}$, $\CB$ has a $\vec b$-labeled edge from $(q_0, q_1, \ldots, q_n)$ to $(q'_0, q'_1, \ldots, q'_n)$ if $q_0\transition{b_0} q'_0$ and if $q_i\transition{(b_{i-1},b_i)} q'_i$ in $\CA_{a_i}$, for each $1\le i\le n$.

Let $u_0,u_1,\ldots, u_n$ be words in $\tilde A^*$ and let $\vec u \in ((\tilde A\cup\{\square\})^{n+1})^*$ be a tuple consisting of these words, padded on the right with sufficiently many $\square$ symbols. Then $\vec u$ labels a path from the tuple of initial states to the state $(s, s_{a_1}, \ldots, s_{a_n})$ if and only if each component of $\vec u$ ends with $\square$, $u_0 \in K$, and $u_i\in L$ and $\ov{u_{i-1}a_i} = \ov{u_i}$ for each $1\le i \le n$.

Let $K'$ be the language accepted by $\CB$, with final state $(s, s_{a_1}, \ldots, s_{a_n})$. To obtain $M_h(K)$, we need to first project $K'$ onto its last component $K''$, and then take the image of $K''$ in the morphism that fixes every letter of $\tilde A$ and maps $\square$ to the empty word.

Let now $\CB'$ be the (non deterministic) automaton obtained from $\CB$ by deleting the $n$ first components of the transition labels: then $\CB'$ accepts the language $K''$. And an automaton for $M_h(K)$ is obtained from $\CB'$ by deleting the $\square$-labeled transitions, and choosing as final, the states from which $(s, s_{a_1}, \ldots, s_{a_n})$ can be reached in $\CB'$ using only $\square$-labeled transitions.
\end{proof}

We can now establish the announced statement.

\begin{prop}\label{prop: decide weak Stallings graph}
There is an algorithm which, given an automatic structure for a group $G$ with set of representatives $L$, given a tuple $(h_1, \ldots, h_k)$ of reduced words and given a finite reduced rooted graph $(\Gamma,1)$, decides whether $(\Gamma,1)$ is Stallings-like for $\langle h_1, \ldots, h_k\rangle$ with respect to $L$.
\end{prop}

\preuve
Let $h_1, \ldots, h_s \in F(A)$, let $H$ be the subgroup of $G$ generated by the $h_i$ and let $K$ be the language of words in $L$ accepted by $(\Gamma,1)$. Then $K$ is rational by construction. Observe that $\Gamma$ is Stallings-like  for $H$ if and only if, for each $u\in K$ and for each $1\le i \le s$, $K$ accepts all the $L$-representatives of $\overline{uh_j}$ and $\overline{uh_j\inv}$.

Lemma~\ref{lemma: compute multiplyer} shows how to find automata for the sets of $L$-representatives of the elements of $\mu(Kh_i)$ and $\mu(Kh_i\inv)$ ($1 \le i \le s$). And it is a classical automata-theoretic algorithm to decide whether a rational language is contained in another: concretely, let $\CA$ be a finite automaton (possibly non-deterministic) with state set $P$ and let $\CB$ be a complete deterministic finite automaton with state $Q$. Then the language accepted by $\CA$ is contained in the language accepted by $\CB$ if and only if, in the product automaton (with state set $P \times Q$), one cannot reach a vertex of the form $(p,q)$ where $p$ is final in $\CA$ and $q$ is not final in $\CB$. This is decided by a breadth-first search of the graph underlying the product automaton.
\eop

\begin{cor}\label{quasi-convex implies weak Stallings graph}
There is a partial algorithm which, given an automatic structure for a group $G$ with language of representatives $L$, and given finitely many reduced words generating a subgroup $H$ of $G$, stops and computes a Stallings-like graph for $H$ if $H$ is $L$-quasi-convex, and never stops if $H$ is not $L$-quasi-convex.
\end{cor}

\preuve
This is a direct application of Corollary~\ref{cor: eventually a weak Stallings graph} and Proposition~\ref{prop: decide weak Stallings graph}.
\eop

\begin{remark}\label{rk: difference with Kapovich}
Kapovich \cite{1996:Kapovich} already established Corollary~\ref{quasi-convex implies weak Stallings graph}, through a different algorithm. In Corollary~\ref{cor: membership problem for quasi-convex}, we draw the same conclusion that he did on the generalized membership problem for $L$-quasi-convex subgroups. It is included here for the sake of completeness. In Section~\ref{sec: computing}, we show how this leads to the computation of  a Stallings graph.
\end{remark}

Let $G$ be an $A$-generated group, equipped with a rational structure $(G,L)$. We say that \emph{$L$-representatives are computable in $G$} if there is an algorithm which, given a reduced word $u\in F(A)$, computes an $L$-representative of $\mu(u)$. We say that \emph{Stallings-like graphs} (resp. \emph{Stallings graphs}) \emph{are computable for $L$-quasi-convex subgroups in $G$} if there is a partial algorithm which, given finitely many reduced words in $F(A)$ generating a subgroup $H$ of $G$, stops and computes a rooted $A$-labeled graph which is Stallings-like for $H$ with respect to $L$ if $H$ is $L$-quasi-convex, and never stops if $H$ is not $L$-quasi-convex.

The \emph{generalized membership problem for subgroups} in $G$ is the following: given finitely many reduced words $h_1,\ldots, h_s, h$, does $h$ belong to the subgroup of $G$ generated by $h_1,\ldots, h_s$? We say that the \emph{generalized membership problem is decidable for $L$-quasi-convex subgroups} in $G$ if there is a partial algorithm which, given the reduced words $h_1,\ldots, h_s, h$, stops and decides whether $h$ belongs to the subgroup $H$ generated by $h_1,\ldots, h_s$, if $H$ is $L$-quasi-convex subgroup, and never stops if $H$ is not $L$-quasi-convex.

We first note the following elementary lemma.

\begin{lemma}\label{lm: weak Stallings graph implies membership problem}
Let $(G,L)$ be a rational structure for $G$ such that $L$-repres\-ent\-atives are computable. Let $H$ be a subgroup of $G$ and let $(\Gamma,1)$ be Stallings-like for $H$ with respect to $L$. There is an algorithm which, given a reduced word $h$, decides whether $\mu(h) \in H$.
\end{lemma}

\preuve
Let $u$ be an $L$-representative of $h$. By definition of a Stallings like graph, $\mu(h) \in H$ if and only if $u$ labels a loop at the base vertex in $\Gamma$. The result follows since we can compute $u$ and check whether $u$ labels a loop at vertex 1.
\eop

Corollary~\ref{quasi-convex implies weak Stallings graph} and Lemma~\ref{lm: weak Stallings graph implies membership problem} directly yield a proof of the following result, first established by Kapovich \cite{1996:Kapovich}.

\begin{cor}\label{cor: membership problem for quasi-convex}
Let $(G,L)$ be a rational structure for $G$, and suppose that $L$-representatives and Stallings-like graphs for $L$-quasi-convex subgroups are computable. Then the generalized membership problem for $L$-quasi convex subgroups is decidable in $G$.

If the rational structure of $G$ is determined by an automatic structure, then the (partial) algorithm solving the generalized membership problem for $L$-quasi convex subgroups is uniform in the automatic structure (that is: there is a partial algorithm which, given an automatic structure for $G$ inducing the rational structure $(G,L)$ and given reduced words $h_1,\ldots, h_s, h$, stops and decides whether $h$ belongs to the subgroup $H$ generated by $h_1,\ldots, h_s$, if $H$ is $L$-quasi-convex subgroup, and never stops if $H$ is not $L$-quasi-convex.
\end{cor}

\subsection{Computing Stallings graphs}\label{sec: computing}

Let $G$ be a group equipped with a rational structure $(G,L)$, where $L$-repr\-es\-ent\-atives are computable. Let $H\le G$ be a quasi-convex subgroup. We saw in Section~\ref{sec: computing a Stallings-like graph} how to compute a Stallings-like graph for $H$ with respect to $L$, under appropriate hypotheses. In this section, we assume that we are given such a graph $(\Gamma,1)$.

By Proposition~\ref{prop: Stallings minimal}, there exists a morphism of labeled rooted graphs $\phi$ from $(\Gamma,1)$ to $(\Schreier(G,H),H)$. The proof of that proposition, together with the solution of the membership problem in $H$ (Lemma~\ref{lm: weak Stallings graph implies membership problem}) actually gives an effective way to compute $\phi(\Gamma)$: if $u_p$ is the label of a path from 1 to vertex $p$ in $\Gamma$, it suffices to identify in $\Gamma$ vertices $p$ and $q$ whenever $\mu(u_pu_q\inv) \in H$, and fold.

The resulting reduced rooted graph, $(\phi(\Gamma),H)$ is again Stallings-like for $H$ with respect to $L$. Indeed, as a subgraph of $\Schreier(G,H)$, its loops at $H$ are all in $\mu\inv(H)$. Moreover, since every $L$-representative of an element of $H$ labels a loop at 1 in $\Gamma$, it also labels a loop at $H$ in $\phi(\Gamma)$.

In particular, $\phi(\Gamma)$ contains $\Gamma_L(H)$. In order to identify $\Gamma_L(H)$, it suffices to verify, for each vertex $p$ of $\phi(\Gamma)$, whether removing it still yields a Stallings-like graph. Proposition~\ref{prop: Stallings minimal} shows that $\Gamma_L(H)$ consists of all the indispensable vertices.

This can be summarized in the following statement.

\begin{prop}
Let $(G,L)$ be a rational structure for $G$ for which $L$-repres\-entatives are computable. If Stallings-like graphs for $L$-quasi-convex subgroups are computable in $G$, then so are Stallings graphs for $L$-quasi-convex subgroups.
\end{prop}

\begin{remark}
If we already have a (fast) algorithm to solve the membership problem in $H$, we can build $\Schreier(G,H)$ edge by edge: we start from the single vertex $H$; then for each graph $\Gamma$ that is constructed, for each vertex $p$ and word $u_p$ labeling a path in $\Gamma$ from $H$ to $p$, and for each letter $a\in \tilde A$, decide whether $u_pau_q\inv \in H$ for some vertex $q$: if so, add an $a$-labeled edge from $p$ to $q$ -- if not, add a new vertex and an $a$-labeled edge from $p$ to that vertex.

We can then run the completion process of Section~\ref{sec: completion process} directly in the graph $\Schreier(G,H)$. We first construct the fragment $\Delta_0$ of $\Schreier(G,H)$ spanned by the loops at $H$ labeled by the generators $h_1,\ldots, h_k$ of $H$. Once a graph $\Delta_i$ is built (within $\Schreier(G,H)$), the graph $\Delta_{i+1}$ is obtained from $\Delta_i$ by reading every relator and every letter at every vertex (still within $\Schreier(G,H)$). It is easily verified that $(\Delta_i,H)$ is the image of $(\Gamma_i,1)$ by the natural morphism (of Proposition~\ref{prop: Stallings minimal}). 

As in Section~\ref{sec: computing a Stallings-like graph}, we can test every $\Delta_i$, to see whether it is Stallings-like for $H$ with respect to $L$.
\end{remark}

\section{Complexity issues}\label{sec: complexity}

In evaluating the complexity of our algorithms, we consider the group $G$ as fixed, and in particular, we have direct access to a finite presentation, an automatic structure with set of representatives $L$, etc.

Our input is the tuple $(h_1,\ldots, h_s)$ of reduced words, such that $H$ is generated by the $\mu(h_i)$. We let $n$ be the sum of the lengths of the $h_i$.

We distinguish in the algorithm several steps, whose time complexity needs to be evaluated. We discuss below:

\begin{itemize}
\item[(a)] the time required to produce $\Gamma_0$, and to produce $\Gamma_{i+1}$ from $\Gamma_i$;

\item[(b)] the time to verify whether $(\Gamma_i,1)$ is Stallings-like for $H$ with respect to $L$; 

\item[(c)] the time required to produce $\Gamma(H)$ from a Stallings-like $\Gamma_i$;

\item [(d)] the number of iterations of the first phase of the algorithm (that is, the number $i$ such that $(\Gamma_i,1)$ is Stallings-like for $H$ with respect to $L$).
\end{itemize}
Let $\rho$ be the sum of the length of the relators of $G$ (the elements of $R$). In our setting, $\rho$ is a constant.

(a) --- It is known that $\Gamma_0$ is produced in time $\O(n\log^*n)$ \cite{2006:Touikan} (that is: in almost linear time). If $\Gamma_i$ has $N$ vertices, then $\Gamma_{i+1}$ has at most $\rho N$ vertices, and it is constructed in time $\O(\rho N \log^*(\rho N))$.

In particular, if $i \ge 1$, then $\Gamma_i$ has at most $\rho^i n$ vertices and the total time to construct it is $\O(i\rho^in\log^*(\rho^in))$.

(b) --- We use the algorithm in the proof of Proposition~\ref{prop: decide weak Stallings graph} to decide whether a reduced rooted graph $(\Gamma,1)$ is Stallings-like for $H$ with respect to $L$. Let $N$ be the number of vertices of $\Gamma$ and let $K$ be the language of all elements of $L$ accepted by $(\Gamma,1)$. An automaton $\CA$ for $K$ is obtained by taking the product of $(\Gamma,1)$ and an automaton accepting $L$ (provided by the automatic structure of $G$): this is computed in time linear in $N$. For each $j \le k$ and $\epsilon \in \{\pm1\}$, we need to perform a breadth-first search on the product automaton of $\CA$  and the automaton accepting $\mu(Kh_j^{\epsilon})$. This product automaton has a number of vertices bounded by some $cN d^{m}$, where $c$ and $d$ are constants (depending on the automatic structure of $G$) and $m = \max_j|h_j|$.  Now breadth-first search takes polynomial time in the number of vertices (in fact time $\O(\ell |A|)$ on an $\ell$-vertex automaton over alphabet $A$ \cite{2009:CormenLeisersonRivest}).  Therefore deciding whether $(\Gamma,1)$ is Stallings-like for $H$ takes time polynomial in $N$ and exponential in $n$.

In particular, the total time required to construct $(\Gamma_i,1)$ and to decide whether it is Stallings-like for $H$, is exponential in $i$ and in $n$.

Note that the exponential dependency in $n$, even with $i$ fixed, is entirely due to the usage of the automatic structure, to verify whether $\Gamma_i$ is Stallings-like.

(c) --- Suppose that we have found out that $(\Gamma,1)$, a reduced rooted graph with $N$ vertices, is Stallings-like for $H$ with respect to $L$. The next step of the computation is to compute the image of $\Gamma$ in $\Schreier(G,H)$ in the canonical graph morphism $\phi$ that maps the base vertex 1 to vertex $H$. This requires:

-  computing for each vertex $p$ of $\Gamma$ a word $u_p$ labeling a path from 1 to $p$: this is done in time $\O(N|A|)$ (again a breadth-first search of the graph $\Gamma$);

- for each pair of vertices $(p,q)$, to decide whether $\mu(u_pu_q\inv) \in H$: this is done by first computing an $L$-representative for $u_pu_q\inv$ (in quadratic time in the length of that word \cite[Thm 2.3.10]{1992:EpsteinCannonHolt}), and then running this $L$-representative in $\Gamma$ (in linear time). Note that $u_pu_q\inv$ has length at most $2k$, where $k$ is the maximal distance from 1 to a vertex in $\Gamma$ ($k$ is a constant of $L$-quasi-convexity for $H$). The total time required is polynomial in $N$.

Once $\phi(\Gamma)$ is computed, we need to verify, for each vertex $p$ of $\phi(\Gamma)$, whether the reduced rooted graph obtained from $\phi(\Gamma)$ by deleting $p$, is still Stallings-like for $H$ with respect to $L$. As discussed in item (b), this takes time polynomial in $N$ and exponential in $n$.

(d) --- The number of iterations of the first phase of the algorithm, that is, the number $i$ such that $(\Gamma_i,1)$ is Stallings-like for $H$ with respect to $L$, can be very large. In fact, we have the following result.

\begin{prop}
Let $G$ be an $A$-generated hyperbolic group, let $h_1,\ldots,h_s \in F(A)$ and let $H = \langle \ov h_1,\cdots, \ov h_s\rangle \le G$ be a quasi-convex subgroup. There exists no computable function of the length of the $h_i$ and $G$, bounding the number of iterations in the first phase of the algorithm to construct the Stallings graph of $H$ with respect to $L$.
\end{prop}

\preuve
If there was such a computable function, we could run the iteration the number of times prescribed by the function, and be guaranteed that the resulting reduced rooted graph $(\Gamma_i,1)$ is Stallings-like for $H$ with respect to $L$, if ever $H$ is $L$-quasi-convex. We could then decide whether $(\Gamma_i,1)$ is indeed Stallings-like for $H$, thus deciding whether $H$ is $L$-quasi-convex.

Yet, there exist hyperbolic subgroups in which it is undecidable whether a given tuple of words generates a quasi-convex subgroup \cite{1982:Rips}. This concludes the proof.
\eop

\section{Applications to algorithmic problems}\label{sec: algorithmic problems}

In this section, we collect a number of applications of our results on the computability of Stallings-like and Stallings graphs for $L$-quasi-convex subgroups. Many of our results have the same structure as Corollary~\ref{cor: membership problem for quasi-convex}: the statement asserts the existence of a partial algorithm $\bbA_{G,L}$ solving a certain problem $\bbP$ on input $X$, relative to a rational structure $(G,L)$ and under certain algorithmic hypothesis (\textit{e.g.}, assuming that Stallings-like graphs are computable for $L$-quasi-convex subgroups). The algorithm is partial in the sense that it terminates and solves the problem $\bbP$ exactly if the subgroups of $G$ determined by input $X$ are $L$-quasi-convex, and it runs forever otherwise. For brevity, we will say that \emph{$\bbA_{G,L}$ is a partial algorithm solving $\bbP$ for $L$-quasi-convex subgroups}.

The statement then asserts that if the rational structure $(G,L)$ is induced by an automatic structure, then this algorithm is \emph{uniform in the automatic structure} in the sense that there exists a single (partial) algorithm which computes $Y$ when given both the automatic structure for $G$ and $X$ as an input. Putting it differently, it means that there exists an algorithm which, given an automatic structure for the group $G$ with language of representatives $L$, outputs the algorithm $\bbA_{G,L}$. Again for brevity, we will simply use the phrase \emph{uniform in the automatic structure}, without re-explaining its meaning in every statement.

\subsection{Membership problem}\label{sec: membership problem}


We first recall Corollaries~\ref{cor: compute constant} and~\ref{cor: membership problem for quasi-convex} (Kapovich \cite{1996:Kapovich}, see also Remark~\ref{rk: least constant of quasi-convexity}).

The \emph{containment} (resp. \emph{equality}) \emph{problem} for subgroups of $G$ is the following: given finite tuples of reduced words generating subgroups $H$ and $K$ of $G$, decide whether $H \le K$ (resp. $H = K$). The \emph{least $L$-quasi-convexity constant problem} is the following: given a finite tuple of words generating a subgroup $H$, compute the least constant of $L$-quasi-convexity of $H$ (this problem makes sense only if $H$ is $L$-quasi-constant).

\begin{prop}\label{prop: membership problem}
Let $G$ be an $A$-generated group, equipped with a rational structure $(G,L)$, in which one can compute $L$-representatives and Stallings-like graphs for $L$-quasi-convex subgroups. 

There are partial algorithms which solve the generalized membership, the containment, the equality and the least constant of $L$-quasi-convexity problems for $L$-quasi-convex subgroups of $G$.
%
%
 
If the rational structure $(G,L)$ is induced by an automatic structure, then these algorithms are uniform in the automatic structure.
\end{prop}

\preuve
The only addition to Corollaries~\ref{cor: compute constant} and~\ref{cor: membership problem for quasi-convex} is the statement on the containment and the equality problems. Given reduced words $h_1, \ldots, h_s, k_1, \ldots, k_t$ in $F(A)$ such that the $h_i$ (resp. the $k_j$) generate an $L$-quasi-convex subgroup $H$ (resp. $K$), one asks whether $H\le K$, or whether $H = K$. The containment problem trivially reduces to the generalized membership problem (applied to $h_i, k_1, \ldots, k_t$ for every $i$), and the equality problem reduces to the containment problem.
\eop

This leads directly to the solution of problems relative to cosets and double cosets. For the latter, we need an additional hypothesis on $L$, which is verified for instance when $L = \LG$.

The \emph{left} (resp. \emph{right}, \emph{double}) \emph{coset computation problem} in $G$ takes as input a reduced word $u$ and a finite tuple of reduced words generating a subgroup $H$ of $G$ (resp. two finite tuples, respectively generating subgroups $H$ and $K$), and outputs an automaton accepting the set of all $L$-representatives of the elements of $H\mu(u)$ (resp. $\mu(u)H$, $H\mu(u)K$). The \emph{double coset equality problem} takes as inputs reduced words $u$ and $v$ and tuples of reduced words generating $H$ and $K$, and decides whether $H\mu(u)K = H\mu(v)K$.

\begin{prop}\label{prop: double coset membership}
Let $G$ be an automatic group and let $(G,L)$ be the resulting rational structure in which $L$-representatives and Stallings-like graphs for $L$-quasi-convex subgroups are computable. Suppose that $L$ is closed under taking inverses.

There are partial algorithms which solve the left, right and double coset computation problems, and the double coset equality problem for $L$-quasi-convex subgroups of $G$.

%
If the rational structure $(G,L)$ is induced by an automatic structure, then these algorithms are uniform in the automatic structure.
\end{prop}

\begin{proof}
Let $u, h_1, \ldots, h_s$ be reduced words and let $H = \langle h_1,\ldots,h_s\rangle$. The partial algorithm in Corollary~\ref{quasi-convex implies weak Stallings graph} computes, exactly if $H$ is $L$-quasi-convex, a reduced rooted graph $(\Gamma,1)$ which is Stallings-like for $H$ with respect to $L$. By definition, the language $\calL(\Gamma)$ accepted by $(\Gamma,1)$ (the set of reduced words which label a loop at vertex 1 in $\Gamma$) contains $L \cap \mu\inv(H)$, and is contained in $\mu\inv(H)$.

Taking the product of $(\Gamma,1)$ with an automaton accepting $L$ yields an automaton accepting exactly the language $\calL(\Gamma) \cap L = L \cap \mu\inv(H)$. We now use Lemma~\ref{lemma: compute multiplyer} to produce an automaton accepting all the $L$-representatives of $\mu((L \cap \mu\inv(H))u) = H\mu(u)$.

By the same process, we can compute an automaton accepting the set of $L$-representatives of $K\mu(u\inv)$. Exchanging the roles of the initial and accepting states in this automaton, and replacing each $a$-labeled transition from state $p$ to state $q$ by an $a\inv$-labeled transition from $q$ to $p$ (for each $a\in \tilde A$) yields a new automaton accepting the set of all $L$-representatives of the inverses of the elements of $K\mu(u\inv)$, that is, the $L$-representatives of the elements of $\mu(u)K$ (this is where we use the assumption that $L = L\inv$).

We now observe that $\mu(v) \in H\mu(u)K$ if and only if $H\mu(v) \cap \mu(u)K \ne \emptyset$, if and only if there is a word which is both an $L$-representative of an element of $H\mu(v)$ and an $L$-representative of an element of $\mu(u)K$. Since we have automata accepting both these sets of $L$-representatives, it is easy to compute an automaton for their intersection, and to verify whether that intersection is non-empty, thus deciding whether $\mu(v) \in H\mu(u)K$.

We conclude after noting that $H\mu(u)K = H\mu(v)K$ if and only if $\mu(u) \in H\mu(v)K$ and $\mu(v) \in H\mu(u)K$.
\end{proof}

\subsection{Finiteness problem}\label{sec: finiteness problem}

Here we consider the \emph{finiteness problem} for subgroups of $G$: given a tuple of reduced words generating a subgroup $H$ of $G$, decide whether $H$ is finite. Again we need a reasonably mild hypothesis on the language $L$ of group representatives (satisfied for instance if $L \subseteq \LG$).

\begin{prop}\label{prop: decide finiteness}
Let $(G,L)$ be a rational structure $(G,L)$ for the group $G$, for which Stallings-like graphs for $L$-quasi-convex subgroups. We also assume that each element of $G$ has a finite number of $L$-representatives. Then there is a partial algorithm which solves the finiteness problem for $L$-quasi-convex subgroups of $G$.

If the rational structure $(G,L)$ is induced by an automatic structure, then this algorithm is uniform in the automatic structure.
\end{prop}

\preuve
As in the proof of Proposition~\ref{prop: double coset membership}, we can construct an automaton accepting the language $L \cap \mu\inv(H)$ (if $H$ is $L$-quasi-convex). Under our hypothesis on the rational structure $(G,L)$, $H$ is finite if and only if $L \cap \mu\inv(H)$ is finite. And it is a classic result of automata theory that one can decide whether the language accepted by a finite automaton, is finite.
\eop

\subsection{Intersection of subgroups}\label{sec: intersection}

The \emph{intersection problem} for subgroups of $G$ takes as input finite tuples of reduced words, generating respectively subgroups $H$ and $K$ of $G$, and outputs a set of generators for the subgroup $H \cap K$. This makes sense only if $H \cap K$ is finitely generated.

Howson's theorem states that the intersection of two finitely generated subgroups of free groups is finitely generated. It is an immediate consequence of the automata-theoretic characterization of $L$-quasi-convexity (Proposition~\ref{prop: Stallings graphs of quasi-convex}) that the intersection of two finitely generated $L$-quasi-convex subgroups is again finitely generated and $L$-quasi-convex (Short \cite{1991:Short}). The following statement follows directly from the computability of Stallings-like graphs for $L$-quasi-convex subgroups. 

\begin{prop}\label{rs}
Let $(G,L)$ be a rational structure for the group $G$, where Stallings-like graphs for $L$-quasi-convex subgroups are computable. There is a partial algorithm which solves the intersection problem for $L$-quasi-convex subgroups of $G$.

If the rational structure $(G,L)$ is induced by an automatic structure, then this algorithm is uniform in the automatic structure.
\end{prop}

\preuve
We have a partial algorithm computing Stallings-like graphs $(\Gamma_H,1)$ and $(\Gamma_K,1)$ for $H$ and $K$, if these subgroups are $L$-quasi-convex. It is easily verified the product graph $\Gamma_H \times_A \Gamma_K$ (or rather, the connected component $\Delta$ of this graph, containing the pair of base vertices) is a reduced rooted graph which is Stallings-like for $H\cap K$. Then $\Delta$ is the Stallings graph of a subgroup of $F(A)$ whose image in $G$ is $H \cap K$. Applying the algorithm in Stallings \cite{1983:Stallings} (namely, constructing a spanning tree of $\Delta$ and producing a generator for each edge of $\Delta$ not in that tree), we compute a generating set of $H \cap K$.
\eop

\begin{remark}\label{rk: rs}
We can also solve the intersection problem for $L$-quasi-convex subgroups using automata accepting $L \cap \mu\inv(H)$ and $L \cap \mu\inv(K)$ (as produced by \cite{1996:Kapovich} for instance): Stallings foldings applied to the product automaton (which accepts $L \cap \mu\inv(H) \cap \mu\inv(K) = L \cap \mu\inv(H\cap K)$), produces the Stallings graph of a subgroup of $F(A)$ which maps to $H\cap K$ in $G$. In other words, a combination of \cite{1996:Kapovich} and \cite{1991:Short} solves the problem. Starting from Stallings-like graphs avoids having the folding stage.

In this context, we note Friedl and Wilton's result on the computability of certain intersections of subgroups of surface groups \cite{2016:FriedlWilton} (Proposition 27, see also Question 30).

Finally, we note that if we start from the Stallings graphs of $H$ and $K$, then the product graph embeds in the Schreier graph of $H \cap K$, and contains the Stallings graph of $H\cap K$.
\end{remark}

\begin{example}\label{ex: computing intersection}
Let $G$ be the modular group, $G = \langle a, b\mid a^2, b^3\rangle$ and $L = \LG$, as in Example~\ref{ex: Stallings graph in PSL}. Let $H = \langle abab\inv, babab\rangle$ and $K = \langle abab, babab\inv\rangle$. Then $H\cap K = \langle (abab\inv)^2, (ab)^3(ab\inv)^3, (ab)^6\rangle$. Figure~\ref{fig: computing intersection} shows the Stallings graphs of $H$ and $K$, each with base vertex 1, as well as their direct product. The latter, with base vertex $(1,1)$, is Stallings-like for $H \cap K$. The Stallings graph of $H$ already appeared in Figure~\ref{fig: Stallings graph in PSL}.
\begin{figure}[htbp]
\centering
\begin{picture}(100,102)(0,-100)
\gasset{Nw=4,Nh=4}
\node(n0)(0.0,-0.0){$1$}

\node(n1)(6.0,-14.0){2}

\node(n2)(0.0,-28.0){3}

\node(n3)(34,-0.0){4}

\node(n4)(28,-14.0){5}

\node(n5)(34.0,-28.0){6}

\drawedge(n0,n1){$b$}

\drawedge(n1,n2){$b$}

\drawedge(n2,n0){$b$}

\drawedge[ELside=r](n3,n4){$b$}

\drawedge[ELside=r](n4,n5){$b$}

\drawedge[ELside=r](n5,n3){$b$}

\drawedge[curvedepth=2.0](n0,n3){$a$}

\drawedge[curvedepth=2.0](n3,n0){$a$}

\drawedge[curvedepth=2.0](n1,n4){$a$}

\drawedge[curvedepth=2.0](n4,n1){$a$}

\drawedge[curvedepth=2.0](n2,n5){$a$}

\drawedge[curvedepth=2.0](n5,n2){$a$}
\node(m1)(72.0,-4.0){$1$}

\node(m2)(78.0,-14.0){2}

\node(m3)(72.0,-24.0){3}

\node(m4)(52,-4.0){4}

\node(m5)(52,-24.0){5}

\node(m6)(98.0,-14.0){6}

\drawedge(m1,m2){$b$}

\drawedge(m2,m3){$b$}

\drawedge(m3,m1){$b$}

\drawedge[curvedepth=2.0](m1,m4){$a$}
\drawedge[curvedepth=2.0](m4,m1){$a$}

\drawedge[ELside=r](m4,m5){$b$}

\drawedge[curvedepth=2.0](m5,m3){$a$}
\drawedge[curvedepth=2.0](m3,m5){$a$}

\drawedge[curvedepth=2.0](m2,m6){$a$}
\drawedge[curvedepth=2.0](m6,m2){$a$}

\drawloop(m6){$b$}
\node(h11)(22.0,-44.0){$\scriptstyle 1,1$}
\node(h22)(22.0,-64.0){$\scriptstyle 2,2$}
\node(h33)(2.0,-64.0){$\scriptstyle 3,3$}

\node(h23)(80.0,-44.0){$\scriptstyle 2,3$}
\node(h31)(100.0,-64.0){$\scriptstyle 3,1$}
\node(h12)(80.0,-64.0){$\scriptstyle 1,2$}

\node(h56)(41.0,-54.0){$\scriptstyle 5,6$}
\node(h46)(61.0,-54.0){$\scriptstyle 4,6$}
\node(h66)(51.0,-70.0){$\scriptstyle 6,6$}

\node(h13)(61.0,-100.0){$\scriptstyle 1,3$}
\node(h21)(41.0,-100.0){$\scriptstyle 2,1$}
\node(h32)(51.0,-84.0){$\scriptstyle 3,2$}

\node(h44)(41.0,-40.0){$\scriptstyle 4,4$}
\node(h55)(61.0,-40.0){$\scriptstyle 5,5$}

\node(h64)(87.0,-76.0){$\scriptstyle 6,4$}
\node(h45)(74.0,-88.0){$\scriptstyle 4,5$}

\node(h54)(28.0,-88.0){$\scriptstyle 5,4$}
\node(h65)(15.0,-76.0){$\scriptstyle 6,5$}

\drawedge(h11,h22){$b$}
\drawedge[ELside=r](h22,h33){$b$}
\drawedge(h33,h11){$b$}

\drawedge(h23,h31){$b$}
\drawedge[ELside=r](h31,h12){$b$}
\drawedge(h12,h23){$b$}

\drawedge(h13,h21){$b$}
\drawedge(h21,h32){$b$}
\drawedge(h32,h13){$b$}

\drawedge[ELside=r](h46,h56){$b$}
\drawedge[ELside=r](h56,h66){$b$}
\drawedge[ELside=r](h66,h46){$b$}

\drawedge[curvedepth=2.0](h22,h56){$a$}
\drawedge[curvedepth=2.0](h56,h22){$a$}

\drawedge[curvedepth=2.0](h12,h46){$a$}
\drawedge[curvedepth=2.0](h46,h12){$a$}

\drawedge[curvedepth=2.0](h32,h66){$a$}
\drawedge[curvedepth=2.0](h66,h32){$a$}

\drawedge[curvedepth=2.0](h11,h44){$a$}
\drawedge[curvedepth=2.0](h44,h11){$a$}
\drawedge(h44,h55){$b$}
\drawedge[curvedepth=2.0](h55,h23){$a$}
\drawedge[curvedepth=2.0](h23,h55){$a$}

\drawedge[curvedepth=2.0](h31,h64){$a$}
\drawedge[curvedepth=2.0](h64,h31){$a$}
\drawedge(h64,h45){$b$}
\drawedge[curvedepth=2.0](h45,h13){$a$}
\drawedge[curvedepth=2.0](h13,h45){$a$}

\drawedge[curvedepth=2.0](h21,h54){$a$}
\drawedge[curvedepth=2.0](h54,h21){$a$}
\drawedge(h54,h65){$b$}
\drawedge[curvedepth=2.0](h65,h33){$a$}
\drawedge[curvedepth=2.0](h33,h65){$a$}
\end{picture}
\caption{\small Computing the intersection of two subgroups of the modular group.}\label{fig: computing intersection}
\end{figure}
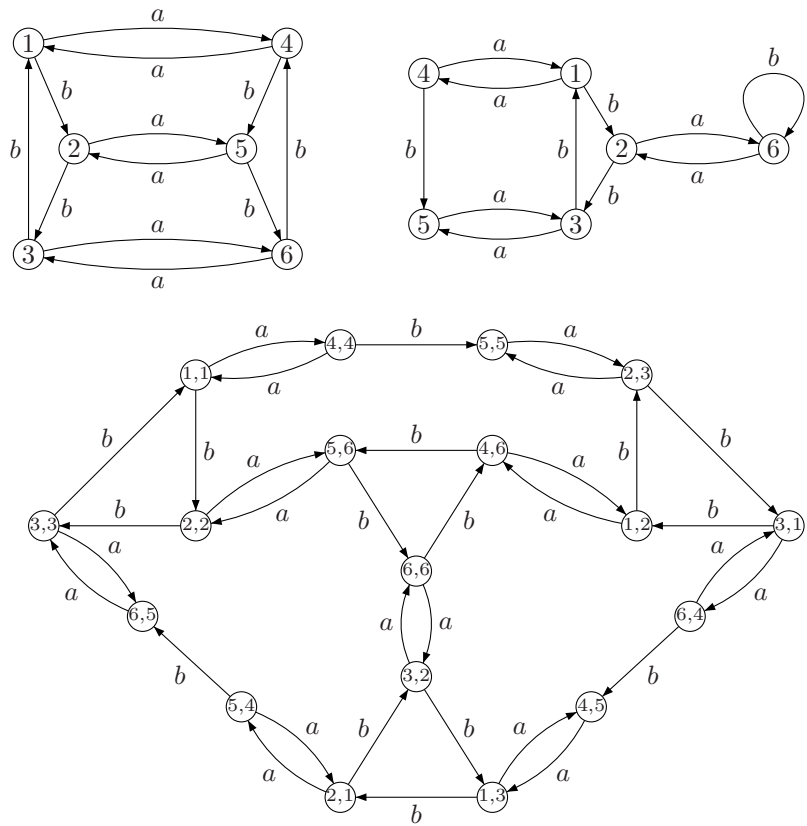
\end{example}

\subsection{Intersections of conjugates of subgroups}\label{sec: BICI}

Following \cite{1998:GitikMitraRips}, the \emph{height} of a subgroup $H$ of $G$ is the maximum number $\hauteur_G(H) = n$ of distinct cosets $Hg_1,\ldots, Hg_n$ such that $\bigcap_{i=1}^nH^{g_i}$ is infinite. The height of a quasi-convex subgroup of a hyperbolic group is finite \cite[Main Theorem]{1998:GitikMitraRips}. 

Recall that a subgroup $H$ of a group $G$ is \emph{almost malnormal} if $H^g \cap H$ is finite for every $g\not\in H$. The subgroup $H$ is \emph{malnormal} if $H^g \cap H = 1$ for every $g\not\in H$: malnormality is equivalent to almost malnormality in a torsion-free group. 

We introduce a technical hypothesis on the regular structure $(G,L)$, essentially a quantitative variant of Hruska and Wise's bounded packing property \cite{2009:HruskaWise}. Under this hypothesis, the height of an $L$-quasi-convex subgroup becomes computable, and a number of interesting problems become decidable, among which the conjugacy and the almost malnormality problem for an $L$-quasi-convex subgroup, as well as an important problem on the computation of intersection of conjugates of two given $L$-quasi-convex subgroups. 

Let $(G,L)$ be a regular structure on a group $G$ and let $\nu\colon \N\to\N$ be a non-decreasing function. We say that \emph{$(G,L)$ satisfies Property $\BP_\nu$} if, whenever $H$ and $K$ are $L$-quasi-convex subgroups of $G$ with constant of $L$-quasi-convexity $k$, if $K, g_1H,\ldots, g_nH$ are pairwise distinct and if $K \cap \bigcap_i H^{g_i}$ is infinite, then there exists an element $z\in G$ such that the ball of center $z$ and radius $\nu(k)$ meets $K$ and each of the $g_iH$.

We will see in Section~\ref{sec: conjugacy in hyperbolic} that Property $\BP_\nu$ is satisfied in particular by the geodesic structure in hyperbolic groups for a linear function $\nu$. In this section, we draw consequences of Property $\BP$ for automatic groups.

\begin{prop}\label{prop: BP to double cosets}
Let $(G,L)$ be a rational structure for a group $G$ and let $\nu$ be a non-decreasing function such that $(G,L)$ satisfy Property $\BP_\nu$. Let $H$ and $K$ be $L$-quasi-convex subgroups of $G$ with constant of $L$-quasi-convexity $k$. Then the following holds.
\begin{itemize}
\item[(1)] Every double coset $HgK$ such that $K \cap H^g$ is infinite, has a representative of length at most $2\nu(k)$.

\item[(2)] There exists a finite family $\calJ$ of infinite intersections $K \cap H^g$ such that any infinite intersection $K \cap H^x$ is conjugated in $K$ to an element of $\calJ$.
\end{itemize}
\end{prop}

\preuve
Let $z$ be the element of $G$ whose existence is asserted by Property $\BP_\nu$, applied with $n=1$ and $g_1 = g\inv$. Then there exist $r\in H$, $r'\in K$, $x$ and $y$ such that $|x|, |y|\le \nu(k)$, $zx = g\inv r$ and $zy = r'$. Then we have:
$$HgK = Hr\inv gr'K = Hx\inv z\inv zyK = Hx\inv yK.$$
The first result follows since $|x\inv y|\le 2\nu(k)$.
Next, we note that $g = rx\inv z\inv = rx\inv y{r'}\inv$, so
$$K\cap H^g = K \cap H^{rx\inv y{r'}\inv} = K \cap H^{x\inv y{r'}\inv} = (K \cap H^{x\inv y})^{{r'}\inv}.$$
Thus we can take, for the set $\calJ$ in the statement, the set of intersections of the form $K \cap H^s$ with $|s| \le 2\nu(k)$ which are infinite.
\eop

This yields a number of decidability or computability results. The \emph{conjugacy containment} (resp. \emph{conjugacy}) \emph{problem} for subgroups of $G$ is the following: given finite tuples of reduced words generating subgroups $H$ and $K$, respectively, decide whether $H$ contains (resp. is equal) to a conjugate of $K$. With the same input, the \emph{intersection of conjugates} problem consists in computing sets of generators for the subgroups in a finite family $\calJ$ as in Proposition~\ref{prop: BP to double cosets} (2) -- assuming that such a family exists. The \emph{height problem} computes $\hauteur_G(H)$ on input a finite tuple of reduced words generating a subgroup $H$ of $G$. The \emph{almost malnormality problem} decides, on the same input, whether $H$ is almost malnormal.

\begin{prop}\label{prop: good properties}
Let $(G,L)$ be a rational structure for a group $G$, for which $L$-representatives and Stallings-like graphs for $L$-quasi-convex subgroups are computable. Let $\nu$ be a computable, non-decreasing function, such that $(G,L)$ satisfy Property $\BP_\nu$. Then there are partial algorithms which solve the conjugacy containment problem, the conjugacy problem and the intersection of conjugates problems for $L$-quasi-convex subgroups of $G$. 

Suppose in addition that each element of $G$ has finitely many $L$-repres\-ent\-atives. There there are partial algorithms which solve the height and the almost malnormality problems for $L$-quasi-convex subgroups of $G$.

If the rational structure $(G,L)$ is induced by an automatic structure, then these algorithms are uniform in the automatic structure and in the function $\nu$.
\end{prop}

\begin{proof}
For all these problems, we first compute a constant of $L$-quasi-convexity $k$ for $H$ and $K$ (resp. for $H$), see Corollary~\ref{cor: compute constant}, and we let $B_k$ be the set of words in $L$ of length at most $2\nu(k)$. 

Note that $H$ contains a conjugate of $K$ if and only if there exists $g\in G$ such that $K \le H^g$. In that case, $H^g \cap K = K$ is infinite, so $HgK$ has an $L$-representative in $B_k$. Therefore $H$ contains a conjugate of $K$ if and only if $K^g \le H$ for some $g\in B_k$, and this can be verified by brute force: for each word $g$ of length at most $2\nu(k)$, we test each generator $x$ of $K$ to verify whether $x^g \in H$, say, by Proposition~\ref{prop: membership problem}. This takes care of the conjugacy containment problem.

Similarly, $H$ and $K$ are conjugated if and only if $H = K^g$ for some $g\in B_k$: we verify, for each $g$ of length at most $2\nu(k)$, and for each generator $x$ of $H$ (resp. $K$) whether $x^{g\inv}\in K$ (resp. $x^g\in H$).

We now turn to the intersection of conjugates problem. As indicated in the proof of Proposition~\ref{prop: BP to double cosets}, the set $\calJ$ can be taken to be the set of infinite intersections of the form $K \cap H^x$, with $|x| \le 2\nu(k)$. This can be computed by Propositions~\ref{rs} and~\ref{prop: decide finiteness}.

Recall that $\hauteur_G(H)$ is the largest $n$ for which there exist $g_1,\ldots,g_n$ such that the $Hg_i$ are pairwise distinct and $\bigcap_{i=1}^nH^{g_i}$ is infinite. Replacing $g_i$ by $g_ig_1\inv$, we see that we can assume that $g_1 = 1$. Property $\BP_\nu$, applied to $K = H$ and the tuple $(g_2,\ldots, g_n)$, states that there exists an element $z\in G$ which is within distance at most $\nu(k)$ from each of the $g_i\inv H$ ($1\le i \le n$).

Thus for every $1 \le i \le n$, there exists $h_i\in H$ and $r_i \in G$ such that $|r_i| \le \nu(k)$ and $zr_i = g_i\inv h_i$. Then $r_i\inv = h_i\inv g_iz$, $Hr_i\inv = Hg_iz$ and $H^{r_i\inv} = H^{g_iz}$. In particular, the cosets $Hr_i\inv$ are pairwise distinct, and $\bigcap_iH^{r_i\inv} = (\bigcap H^{g_i})^z$ is infinite.

So, to compute $\hauteur_G(H)$, we first compute a list $(Hx_1,\ldots,Hx_N)$ of the cosets of $H$ with a representative $x_i$ of length at most $\nu(k)$ (by deciding membership in $H$ of all products $x\inv y$ of words of at most that length, see Proposition~\ref{prop: membership problem}). We then compute all the subsets $\{i_1,\ldots,i_s\}$ of $\{1,\ldots,N\}$ for which $\bigcap_{j=1}^sH^{x_{i_j}}$ is infinite (Propositions~\ref{rs} and~\ref{prop: decide finiteness}). The height of $H$ is the maximal cardinality of such a subset.

Finally, consider the almost malnormality problem. Let $g\not\in H$. By Proposition~\ref{prop: BP to double cosets} (1), if $H \cap H^g$ is infinite, then $HgH = HxH$ with $|x|\le 2\nu(k)$, and we note that $x\not\in H$ and $H\cap H^x$ is a conjugate of $H\cap H^g$. Thus, $H$ is almost malnormal if and only if, for each $x\not\in H$ with $|x|\le 2\nu(k)$, the intersection $H \cap H^x$ is finite, which can be effectively verified by Propositions~\ref{rs} and~\ref{prop: decide finiteness}.
\end{proof}

\subsection{Intersection of conjugates of subgroups: the hyperbolic case}\label{sec: conjugacy in hyperbolic}

We apply the results of Section~\ref{sec: BICI} to the quasi-convex subgroups of hyperbolic groups. First we follow the same reasoning as in \cite[Sec. 4 and 8]{2009:HruskaWise} to prove the following result.

\begin{prop}\label{prop: hyperbolic BP_nu}
Let $G$ be a $\delta$-hyperbolic group. Then $(G,\LG)$ satisfies Property $\BP_\nu$, for the function $\nu(k) = k + 2\delta$.
\end{prop}

\preuve
We rely on a few well-known facts on hyperbolic groups (see \textit{e.g.} \cite[Thm 2.28, Prop. 4.2]{2002:KapovichBenakli}: an infinite subgroup of $G$ contains an element of infinite order; if $x$ has infinite order, then the limit $c$ of the geodesic segments $[x^{-m};x^m]$ ($m\ge 1$) is a bi-infinite geodesic line with two endpoints in $\partial G$, namely, $x^{-\infty} = \lim_m x^{-m}$ and $x^\infty = \lim_m x^m$; if $c$ and $c'$ are bi-infinite geodesic lines with the same endpoints, then every vertex of $c$ is within distance $2\delta$ of a vertex of $c'$.

Let $H,K$ be quasi-convex subgroups, with constant of quasi-convexity $k$, and let $g_1,\ldots, g_n\in G$ be such that $K$ and the $g_iH$ are pairwise distinct and $K \cap \bigcap_iH^{g_i}$ is infinite. Let $x$ be an element of infinite order of $K \cap \bigcap_iH^{g_i}$ and let $c$ be the bi-infinite geodesic described above. Since $K$ is quasi-convex with constant $k$, every vertex of $c$ is within distance at most $k$ from $K$.

Let $1 \le i \le n$. Consider the element $g_i\inv xg_i \in H$, also of infinite order: there exists a bi-infinite geodesic line $c_i$ with endpoints $\lim g_i\inv x^{-m}g_i$ and $\lim g_i\inv x^mg_i$, which remains within distance at most $k$ from $H$. Now $\lim g_i\inv x^mg_i = g_i\inv \lim x^mg_i$, and each $x^mg_i$ is within distance $|g_i|$ from $x^m$. So $\lim x^mg_i = x^\infty$ and the endpoints of $c_i$ are $g_i\inv x^{-\infty}$ and $g_i\inv x^\infty$.

Acting on the left by $g_i$ yields a geodesic line $g_i c_i$, with endpoints $x^{-\infty}$ and $x^\infty$, within distance at most $k$ from $g_iH$. As noted above, since  the lines $c$ and $g_i c_i$ ($1\le i \le n$) have the same endpoints, they remain within distance at most $2\delta$ from one another.

Let $z$ be a point on $c$. Then the ball of center $z$ and radius $k$ meets $K$, and the ball of center $z$ and radius $k+2\delta$ meets every coset $g_i H$. This concludes the proof.
\eop

In view of Proposition~\ref{prop: good properties}, this yields the following corollary.

\begin{cor}\label{cor: good properties}
Let $G$ be a hyperbolic group. There are partial algorithms which solve the conjugacy containment, the conjugacy, the intersection of conjugates, the height and the almost malnormality problems for quasi-convex subgroups.

These algorithms are uniform in the finite presentation of $G$.
\end{cor}

\preuve
Propositions~\ref{prop: hyperbolic BP_nu} and~\ref{prop: good properties} prove the existence of the announced partial algorithms. The uniformity statement follows from the fact that there is a partial algorithm which, given a finite group presentation, halts exactly if the group is hyperbolic, and outputs an automatic structure for that group, see \cite{2001:EpsteinHolt}.
\eop

\begin{remark}
Note that it is undecidable whether a tuple of elements generates a malnormal subgroup of a hyperbolic group \cite{2001:BridsonWise}. Here of course, the input subgroup is required to be quasi-convex.
\end{remark}

\subsection{Deciding the finite index property}\label{sec: finite index}

In order to get a decidability statement for the finite index property, we make an additional assumption on the rational structure $(G,L)$ -- strongly inspired from Silva, Soler-Escriva and Ventura \cite{2016:SilvaSoler-EscrivaVentura}. We say that $L$ has the \emph{extendability property} if the following holds: for every $u\in L$, there exists a sequence $(v_n)_n$ such that, for each $n$, $uv_n\in L$ and for almost all $m$, $u$ is a prefix of some $L$-representative of $\mu(uv_nv_m\inv u\inv)$.

\begin{theorem}\label{thm: characterize finite index}
Let $(G,L)$ be a rational structure for a group $G$, let $H \le G$ be an $L$-quasi-convex subgroup and let $(\Gamma,1)$ be a reduced rooted graph which is Stallings-like for $H$ with respect to $L$.
\begin{itemize}
\item[(1)] If every word of $L$ labels a path starting at the base vertex in $\Gamma$, then $H$ has finite index, at most equal to the number of vertices of $\Gamma$.

\item[(2)] If every word of $L$ labels a path starting at the base vertex in the Stallings graph $\Gamma_L(H)$, then $H$ has finite index, equal to the number of vertices of $\Gamma_L(H)$.

\item[(3)] Suppose now that $L$ has the extendability property. If $H$ has finite index, then every word of $L$ labels a path starting at the base vertex in $\Gamma$.
\end{itemize}
\end{theorem}

\preuve
Observe that if $u$ and $v$ label paths in $\Gamma$, from the base vertex to the same vertex $p$, then the word $uv\inv$ labels a loop at the base vertex, and hence $\mu(uv\inv) \in H$. In particular, $H\mu(u) = H\mu(v)$. Thus there is a well-defined partial map, from the set of vertices of $\Gamma$ onto the set of cosets $H\mu(u)$ such that there is an $u$-labeled path starting at the base vertex of $\Gamma$.

If every word of $L$ labels a path starting at the base vertex in $\Gamma$, this implies the existence of an onto partial map from the set of vertices of $\Gamma$ to the set of all the $H$-cosets, that is, $H$ has finite index, at most equal to the number of vertices of $\Gamma$. This establishes (1).

In the particular case where $(\Gamma,1) = (\Gamma_L(H),H)$, this shows that the index of $H$ is less than or equal to the number of vertices of $\Gamma_L(H)$. But that graph is contained in $\Schreier(G,H)$, which has as many vertices as there are cosets of $H$. Statement (2) follows. 

Now, assume that $L$ has the extendability property, and that $H$ has finite index in $G$. Let $u\in L$. By the extendability property of $L$, there exists a sequence of words $v_n$ such that, for each $n$, $uv_n \in L$ and for almost all $m$, $u$ is a prefix of some $L$-representative of $\mu(uv_nv_m\inv u\inv)$. Since $H$ has finite index, the sequence $H\mu(uv_n)$ takes only finitely many values. In particular, there exists $n$ such that for infinitely many values of $m$, we have $H\mu(uv_n) = H\mu(uv_m)$. In particular, for some $n,m$, $u$ is a prefix of an $L$-representative $w$ of $\mu(uv_nv_m\inv u\inv)$, with $\mu(uv_nv_m\inv u\inv) \in H$. Since $w$ labels a loop at the base vertex in $\Gamma$, it follows that $u$ labels a path starting at the base vertex.
\eop

It is an elementary property of rational languages that one can decide whether a rational language is contained in another. Therefore Theorem~\ref{thm: characterize finite index} yields the following corollary. The \emph{finite index problem} in $G$ computes the index of the subgroup of $G$ generated by a given tuple of reduced words.

\begin{cor}
Let $(G,L)$ be a rational structure for a group $G$, satisfying the extendability property, and for which $L$-representatives are computable and Stallings-like graphs are computable for $L$-quasi-convex subgroups. Then there is a partial algorithm which solves the finite index problem for $L$-quasi-convex subgroups of $G$.
%
%
\end{cor}


\begin{example}\label{ex: index of subgroups of modular group}
We use, again, the example of the modular group $G = \langle a,b \mid a^2, b^3\rangle$, see Examples~\ref{ex: Stallings graph in PSL} and~\ref{ex: computing intersection}. Recall that $\LG$ is the set of reduced words over alphabet $\{a,b\}$ without consecutive occurrences of $a$, $a\inv$, $b$ or $b\inv$.

We first show that $\LG$ satisfies the extendability property: if $u$ ends with an $a$ or $a\inv$, say $u = xa$, we let $v_n = b\inv(ab)^n$. Then each $uv_n \in \LG$ and, if $m > n$, we have
\begin{align*}
uv_nv_m\inv u\inv &=_G xab\inv(ab)^n(b\inv a\inv)^mba\inv x \\
&=_G xab^{-2}a\inv(b\inv a\inv)^{m-n-1}ba\inv x \\
&=_G xaba\inv(b\inv a\inv)^{m-n-1}ba\inv x \in\LG.
\end{align*}
If instead $u$ ends with $b$ or $b\inv$, we let $v_n = ab\inv(ab)^n$.

We can therefore use Theorem~\ref{thm: characterize finite index} to assert that the subgroup $H$ in Examples~\ref{ex: Stallings graph in PSL} and~\ref{ex: computing intersection} has index 6, and that the subgroup $K$ in Example~\ref{ex: computing intersection} has infinite index. In fact, the cosets $Ka(b\inv a)^n$ are pairwise distinct.
\end{example}

\section{Relatively quasi-convex subgroups in relatively hyperbolic groups}\label{sec: relatively hyperbolic}

In this section we prove various algorithmic results on relatively quasi-convex subgroups of certain relatively hyperbolic groups, which generalize the results from the previous section on quasi-convex subgroups of hyperbolic groups. 

\subsection{Preliminaries on relatively hyperbolic groups}

There are several definitions of relative hyperbolicity, introduced by Gromov \cite{1987:Gromov}, Farb \cite{1998:Farb}, Bowditch \cite{2012:Bowditch}, Dru\c tu and Sapir \cite{2005:DrutuSapir}, Osin \cite{2006:Osin}, which turn out to be equivalent (see Bumagin \cite{2005:Bumagin}, Dahmani \cite{2003:Dahmani}, Hruska \cite[Theorem 5.1]{2010:Hruska}).

We will use the following definition: let $G$ be a group generated by a finite set $A$, and let $\calP$ be a \emph{peripheral structure}, that is, a finite collection $\calP$ of finitely generated subgroups in $G$, called the \emph{peripheral} subgroups. Let $B = \bigcup _{P\in \calP} (P \setminus\{1\})$ and let $\Cayley(G,A\cup B)$ the \emph{relative Cayley graph} of $G$. That is: the non trivial elements of the subgroups $P_i$ are considered as alphabet letters as well as the set of generators $A$.

If $\calP = \{P_1, \ldots, P_\ell\}$, let $\hat F=F(A)\ast P_1\ast\ldots\ast P_\ell$. A word on alphabet $\tilde A \cup B$ is \emph{$\hat F$-reduced} if it is reduced in the sense of free products, that is, if it is of the form $x_0p_1x_1\cdots p_mx_m$ where the $x_j$ are reduced words in $F(A)$, the $p_j$ are elements of $B$, and $p_j$ and $p_{j+1}$ do not sit in the same $P_i$ if $x_j = 1$. An $\hat F$-reduced word is a \emph{relative (quasi-) geodesic} if it is (quasi-) geodesic in the Cayley graph $\Cayley(G,A\cup B)$.

We say that the group $G$ is \emph{weakly relatively hyperbolic, relative to the peripheral structure $\calP$}, if $\Cayley(G,A\cup B)$ is a hyperbolic metric space \cite[Def. 2.1]{2009:Martinez-Pedroza}.

If $p$ is a path in $\Cayley(G,A\cup B)$ and $P\in \calP$, we say that a subpath $p'$ is a \emph{$P$-syllable} if all the edges in $p'$ are labeled by an element of $P$. A \emph{$P$-component} of $p$ is a maximal $P$-syllable. We say that two vertices $x,y\in G$ are \emph{$P$-connected} if $y \in xP$, and that two $P$-syllables are $P$-connected if they contain $P$-connected vertices.

The \emph{bounded coset penetration property} (BCP) is defined as follows.

\begin{definition} \cite[Definition 6.5]{2006:Osin}\label{defn: bcp}
The pair $(G,\calP)$ has the BCP property, if for any $\lambda \geq 1$,
there exists a constant $a = a(\lambda)$ such that the following conditions hold. Let $p, q$ be $(\lambda, 0)$-quasi-geodesics without backtracking in $\Cayley(G, A\cup B)$ (that is: for each $x\in G$ and $P\in \calP$, all the vertices of $p$ (resp. $q$) in $xP$ sit in the same $P$-component) with the same initial vertex and whose terminal vertices $p_+$ and $q_+$ satisfy $d_{A}(p_+,q_+) \leq 1$. Then for each $P\in \calP$:
 
1) if $p$ has a $P$-component $s$, from vertex $s_-$ to vertex $s_+$, and if $d_A(s_-,s_+) \geq a$, then $q$ has a $P$-component which is $P$-connected to $s$;
 
2) if $p$ and $q$ both have $P$-components, from vertices $s_-$ and $t_-$ to vertices  $s_+$ and $t_+$ respectively, then $d_A(s_-,t_-) \leq a$ and $d_A(s_+,t_+) \leq a$.
\end{definition} 

Then $G$ is \emph{relatively hyperbolic, relative to the peripheral structure $\calP$} if $G$ is weakly relatively hyperbolic relative to $\calP$ and the pair $(G, \calP)$ has the BCP property. For short, we say that $(G,\calP)$ is a relatively hyperbolic group.

\subsection{Relatively quasi-convex subgroups}\label{sec: slender case}

Let $G$ be an $A$-generated group equipped with a peripheral structure $\calP$. A subset $S \subseteq G$ is \emph{relatively  quasi-convex} if there exists a constant $k$ such that, if an $\hat F$-reduced word labels a relative geodesic path $\pi$ in $\Cayley(G,A\cup B)$ between two elements of $S$, then $\pi$ lies in the $k$-neighborhood of $S$ in $\Cayley (G,A)$: that is, for any vertex $x\in \pi$, there exists a vertex $y\in S$ such that $\dist_A(x,y)\leq k$ --- where $\dist_A(x,y)$ denotes the distance between $x$ and $y$ in $\Cayley(G,A)$, namely $\dist_A(x,y) = |x\inv y|_A$.

We first show a general result on relatively quasi-convex subgroups satisfying a certain finite generation property. We say that a subgroup $H \le G$ is \emph{peripherally finitely generated} if $P \cap H^g$ is finitely generated for every peripheral group $P \in \calP$ and every $g\in G$.

\begin{prop}\label{prop: slender case}
Let $G$ be a weakly relatively hyperbolic group relative to a peripheral structure $\calP$ and let $H$ be a relatively quasi-convex subgroup of $G$ which is peripherally finitely generated. Then there exists a finite fragment $\Gamma(H,\calP)$ of the Schreier graph $\Schreier(G,H)$ relative to $A$, with the following property: every element of $H$ admits a representative which labels a loop at vertex $H$ in $\Gamma(H,\calP)$. 

In particular $H$ is finitely generated.
\end{prop}

\begin{remark}\label{rk: slender}
Proposition~\ref{prop: slender case} extends several known results. It applies in particular if $H$ is \emph{strongly relatively quasi-convex}, that is, if $P \cap H^g = 1$ for every $P \in \calP$ and $g\in G$: in that case, Osin proved that $H$ is finitely generated and hyperbolic \cite[Thms 4.13 and 4.16]{2006:Osin}. It also applies if the peripheral structure $\calP$ consists of so-called \emph{slender} groups, that is, groups in which every subgroup is finitely generated (\textit{e.g.} finitely generated abelian groups): Hruska showed that this condition on $\calP$ is equivalent to the fact that all relatively quasi-convex subgroups are finitely generated \cite[Cor. 9.2]{2010:Hruska}.
\end{remark}

We will use the following elementary fact.

\begin{fact}\label{fact: single coset}
Let $H, P$ be subgroups of a group $G$. For all $g, g'\in G$, the set of elements $p\in P$ such that $Hgp = Hg'$, namely the set $P \cap g\inv Hg'$, consists of a single coset of $P\cap H^g$ in $P$.
\end{fact}

\proofof{Proposition~\ref{prop: slender case}}
Let $k$ be a quasi-convexity constant for $H$. Let $Hg$, $Hg'$ be cosets of $H$ and let $P \in \calP$. If $P \cap g\inv Hg' \ne \emptyset$, we fix a word $z_P(Hg,Hg') \in F(A)$ such that $\bar z_P(Hg,Hg') \in P \cap g\inv Hg'$. If $P \cap H^g \ne 1$, we also fix a finite set of words $Y_{P,Hg}$ in $F(A)$, whose images in $G$ generate $P \cap H^g$: such a set exists since $H$ is peripherally finitely generated.

Then, by Fact~\ref{fact: single coset}, every element $p \in P$ such that $Hgp = Hg'$ (that is: every element of $P \cap g\inv Hg'$) is the image in $G$ of a word of the form $yz_P(Hg,Hg')$ for some $y \in (Y_{P,Hg}\cup Y_{P,Hg}\inv)^*$.

We now describe the construction of $\Gamma(H,\calP)$, starting from $\Schreier_k(G,H)$, which is the subgraph of $\Schreier(G,H)$ consisting of the vertices $Hg$ such that $|g|_A \le k$ (see Section~\ref{sec: quasi-convex subgroups}). For each vertex $Hg$ of $\Schreier_k(G,H)$ and for each $P \in \calP$, we add to $\Schreier_k(G,H)$:
\begin{itemize}
\item the loops in $\Schreier(G,H)$ at $Hg$ labeled by a word in $Y_{P,Hg}$ (if $P\cap H^g \ne 1$);
\item and for every vertex $Hg'$ of $\Schreier_k(G,H)$ such that $P \cap g\inv Hg' \ne \emptyset$, the path in $\Schreier(G,H)$ labeled by $z_P(Hg,Hg')$, from $Hg$ to $Hg'$.
\end{itemize}
By construction, the resulting reduced graph $\Gamma(H,\calP)$, rooted at $H$, is a finite subgraph of $\Schreier(G,H)$ containing $\Schreier_k(G,H)$. As a result, the image in $G$ of the label of a loop at $H$ in $\Gamma(H,\calP)$, is an element of $H$.

Let now $g\in H$ and let $w = x_0p_1x_1p_2 \cdots p_mx_m$ be a $\hat F$-reduced word (with each $x_i \in F(A)$ and each $p_i$ in some peripheral subgroup), which labels a relative geodesic path in $\Cayley(G,A\cup B)$, from 1 to $g$. For each $j$, let $g_j = \mu(x_0p_1\cdots x_{j-1}p_j)$ ($g_0 = 1$). Since $H$ is relatively quasi-convex with constant $k$, every $g_j$ and $g_j\bar x_j$ is within $A$-distance at most $k$ from a vertex of $H$, and in particular, the cosets $Hg_j$ and $Hg_j\bar x_j$ are vertices of $\Schreier_k(G,H)$.

By construction of $\Gamma(H,\calP)$, the word $x_j$ labels a path in $\Schreier_k(G,H)$, and hence also in $\Gamma(H,\calP)$, from $Hg_j$ to $Hg_j\bar x_j$. By our earlier observation, if $p_j \in P_j$ (with $P_j\in \calP)$, then $p_j \in P_j \cap \bar x_{j-1}\inv g_{j-1}\inv H g_j$ so there exists a word $y_j \in (Y_{P_j,Hg_{j-1}\bar x_{j-1}}\cup Y_{P_j,Hg_{j-1}\bar x_{j-1}}\inv)^*$ such that $p_j = \bar y_j\bar z_{P_j}(Hg_{j-1}\bar x_{j-1},Hg_j)$. By definition of $\Gamma(H,\calP)$, $y_jz_{P_j}(Hg_{j-1}\bar x_{j-1},Hg_j)$ labels a path in that graph from $Hg_{j-1}\bar x_{j-1}$ to $Hg_j$.

Thus the word
$$\tilde w = x_0y_1z_{P_1}(Hg_0\bar x_0,Hg_1)x_1 \cdots y_mz_{P_m}(Hg_{m-1}\bar x_{m-1},Hg_m)x_m,$$
over alphabet $A$, labels a loop at $H$ in $\Gamma(H,\calP)$ and it has the same image in $G$ as $w$: that is, it is a representative of $\bar w$.
\eopo

\subsection{Stallings graphs for relatively quasi-convex subgroups}

We now use a variant of the construction in Proposition~\ref{prop: slender case} to discuss certain relatively quasi-convex subgroups of relatively hyperbolic groups, relative to a peripheral structure $\calP$ consisting of geodesically bi-automatic groups. We introduce an additional hypothesis on the subgroup $H$: we say that it has \emph{peripherally finite index} if, for each peripheral subgroup $P$ and each $g\in G$, $P \cap H^g$ is either finite or a finite index subgroup of $P$ \footnote{These subgroups are called \emph{fully quasi-convex} by Dahmani \cite{2003:Dahmani-2} and Mart\'\i nez-Pedroza \cite{2009:Martinez-Pedroza}.}. Note that strongly relatively quasi-convex subgroups (see Remark~\ref{rk: slender}) have peripherally finite index.

\begin{theorem}\label{thm: peripherally finite index}
Let $(G,\calP)$ be an $A$-generated relatively hyperbolic group, such that $\calP$ consists of geodesically bi-automatic groups. Then one can compute a finite set $X$ containing $A$, a geodesically bi-automatic structure for $G$ over alphabet $X$ and a constant $\alpha > 0$ such that, if $L$ is the set of representatives of the elements of $G$ given by the bi-automatic structure, the following properties hold:
\begin{itemize}
\item every $L$-quasi-convex subgroup is relatively quasi-convex;

\item every relatively quasi-convex subgroup with peripherally finite index, is $L$-quasi-convex.
\end{itemize}
\end{theorem}

We make use of several fundamental ingredients, the first of which is a set of results due to Antol\'{\i}n and Ciobanu \cite{2016:AntolinCiobanu} (more specifically Corollary 1.9, Lemma 5.3 and Theorem 7.6 in \cite{2016:AntolinCiobanu}), which states essentially that we can enlarge the alphabet $A$ in such a way that $G$ admits a geodesically bi-automatic structure extending given bi-automatic geodesic structures of the peripheral subgroups. If $w$ is a reduced word in $F(A)$, the \emph{$\calP$-factorization} of $w$ \cite[Construction 4.1]{2016:AntolinCiobanu} is the uniquely determined factorization $w = x_0p_1x_1\cdots p_mx_m$, where each $x_i \in F(A \setminus B)$, each $p_i$ is a non-empty word in $F(A \cap P)$ for some $P \in \calP$ and, if $x_i = 1$ and $a$ is the first letter of $p_{i+1}$, then the letters of $p_ia$ do not all sit in the same $P$ ($P\in\calP$). In that case, we denote by $\hat w$ the ($\hat F$-reduced) word $\hat w = x_0 \bar p_1x_1\cdots \bar p_mx_m \in \hat F$. In addition, suppose that for each $P\in \calP$, $L_P$ is a language of reduced words in $F(A \cap P)$: we denote by $\Rel(A,(L_P)_{P\in\calP})$ the set of reduced words $w\in F(A)$ whose $\calP$-factorization $w = x_0p_1x_1\cdots p_mx_m$ is such that, if $p_i\in F(A\cap P)$ ($P\in\calP$), then $p_i\in L_P$. We also denote by $\LG(A,(L_P)_{P\in\calP})$ the set of $A$-geodesics in $\Rel(A,(L_P)_{P\in\calP})$.

\begin{prop}\label{prop: Antolin-Ciobanu}
Let $(G,\calP)$ be a relatively hyperbolic $A$-generated group and let $B = \bigcup_{P\in \calP}P \setminus\{1\}$. One can compute $\lambda\ge 1$, $\epsilon \ge 0$, a finite set $A' \subseteq B$ and an explicit correspondance $\Psi$ with the following properties:
\begin{itemize}
\item if $X$ is a finite set such that $A \cup A' \subseteq X \subset A\cup B$ and if, for each $P\in \calP$, we are given $L_P$ is the set of representatives of the elements of $P$ in a geodesically bi-automatic structure for $(P,X\cap P)$, then the correspondence $\Psi$ produces a geodesically bi-automatic structure for $(G,X)$, where the language of representatives of the elements of $G$ is $\LG(X,(L_P)_{P\in\calP})$;

\item If $w\in F(X)$ labels a geodesic from 1 to $\bar w$ in $\Cayley(G,X)$ and if $\bar w \in P$ for some peripheral subgroup $P\in \calP$, then $w \in F(X \cap P)$. In general, $\hat w$ is a relative $(\lambda,\epsilon)$-quasi-geodesic from 1 to $\bar w$ in $\Cayley(G,A\cup B)$.
\end{itemize}
\end{prop}

\begin{remark}
The computability statement for the bi-automatic structures that determine the rational structure $(G,\LG(X,(L_P)_{P\in\calP})$ in Proposition~\ref{prop: Antolin-Ciobanu}, is established at the end of \cite[Section~1]{2016:AntolinCiobanu}, and relies in part on results of Dahmani \cite{2009:Dahmani}.
\end{remark}

%
%
%

We also record a result of Osin which refines the BCP property \cite[Thm 3.23]{2006:Osin}, see also \cite[Thm 2.8]{2016:AntolinCiobanu}.

\begin{prop}\label{prop: improved BCP} 
Let $(G,\calP)$ be a relatively hyperbolic $A$-generated group, $\lambda \ge 1$ and $\epsilon > 0$. One can compute a constant $a(\lambda,\epsilon) > 0$ such that, if $p$ and $q$ are two relative $(\lambda,\epsilon)$-quasi-geodesics with the same endpoints and which do not backtrack, then every vertex of $p$ is at distance at most $a(\lambda,\epsilon)$ from some vertex of $q$ in $\Cayley(G,A)$.
\end{prop}

\proofof{Theorem~\ref{thm: peripherally finite index}}
Let $\lambda, \epsilon$ be given by  Proposition~\ref{prop: Antolin-Ciobanu} and let $\alpha = a(\lambda,\epsilon)$ be given by Proposition~\ref{prop: improved BCP}.

Let $A'$ be the set given by Proposition~\ref{prop: Antolin-Ciobanu}. For each peripheral subgroup $P\in \calP$, let $Y_P$ be a generating set for $P$ containing $(A\cup A')\cap P$, and let $L_P$ be a geodesically bi-automatic structure for $P$ over alphabet $Y_P$. Finally, let $X = A \cup A' \cup \bigcup_{P\in\calP} Y_P$. By Proposition~\ref{prop: Antolin-Ciobanu}, $L = \LG(X,(L_P)_{P\in\calP})$ is a geodesically bi-automatic structure for $(G,X)$.

Let $g \in G$, let $\pi$ be a relative geodesic from 1 to $g$ in $\Cayley(G,A\cup B)$, and let $w \in F(X)$ be an $L$-representative of $g$. By Proposition~\ref{prop: Antolin-Ciobanu}, $\hat w$ is a relative $(\lambda,\epsilon)$-quasi-geodesic from 1 to $g$, and by Proposition~\ref{prop: improved BCP}, the path labeled by $\hat w$ and $\pi$ are within $A$-distance $\alpha$ from each other.

If $H$ is $L$-quasi-convex with constant $k$, then every vertex along $\pi$ is at $A$-distance at most $\alpha$ from a vertex of the path labeled $\hat w$, which itself is at distance at most $k$ of a vertex in $H$. Therefore $H$ is relatively quasi-convex with constant $k+\alpha$.

Let us now assume that $H$ is relatively quasi-convex with constant $k$.
For every peripheral subgroup $P$ and every $g\in G$ such that $|g|_A \le k+\alpha$ and $P\cap H^g$ is infinite, we denote by $\Gamma(P,Hg)$ the (finite) Schreier graph of $P\cap H^g$ in $P$ over alphabet $Y_P$. Because of the finite index property of $P\cap H^g$ in $P$, every word in $F(Y_P)$ can be read in $\Gamma(P,Hg)$, starting at any vertex. We note that if $u, v \in P$ are such that $(P\cap H^g)u = (P \cap H^g)v$, then $vu\inv \in P\cap H^g$, so $gvu\inv g\inv \in H$ and hence $Hgu = Hgv$. It follows that there is a graph morphism $\phi_{P,Hg}$ from $\Gamma(P,Hg)$ to $\Schreier(G,H)$ (over alphabet $X$), mapping vertex $P\cap H^g$ to vertex $Hg$ (this morphism can be shown to be injective).

Suppose now that $P\in \calP$, $g\in G$ and $P \cap H^g$ is finite. For each $p\in P$, the set of elements $q\in P$ such that $Hgq = Hgp$ is a coset of $P \cap H^g$ (see Fact~\ref{fact: single coset}), namely $(P \cap H^g)p = P \cap g\inv Hgp$, and hence a finite set.

Then we let $\Gamma$ be the finite fragment of $\Schreier(G,H)$ consisting of:
\begin{itemize}
\item the vertices at $A$-distance at most $\alpha+k$ from $H$ and the $A$-labeled edges between them;
\item the images of the morphisms $\phi_{P,Hg}$ where $P\in \calP$, $g\in G$ and $P\cap H^g$ is infinite;
\item and for each $P\in \calP$ and $g, g'\in G$ such that $|g|_A, |g'|_A \le \alpha+k$, $P\cap H^g$ is finite and $P \cap g\inv Hg' \ne \emptyset$, all the (finitely many) paths of $\Schreier(G,H)$, starting from $Hg$, labeled by a word $z \in L_P$ such that $Hg\bar z = Hg'$.
\end{itemize}
Now let $w\in L$ be a representative of an element of $H$ and let
$$w = x_0p_1x_1\cdots p_mx_m$$
be its $\calP$-factorization. Since each $x_i$ is written on alphabet $X \setminus B$, and since $A' \subseteq B$, it is in fact written on alphabet $A \setminus B$. Moreover, by definition of $L$, each $p_i$ is in some $L_P$ ($P\in \calP$).

As noted above, every vertex along the path labeled by $\hat w = x_0\bar p_1x_1\cdots \bar p_mx_m$ from 1 to $g$ is at $A$-distance at most $\alpha$ from a relative geodesic from 1 to $g$, and hence at distance at most $k+\alpha$ from $H$.

Let $0\le i \le m$ and $g_i = \mu(x_0p_1x_1\cdots x_{i-1}p_i)$ (with $g_0 = 1$). Then $Hg_i$ and $Hg_i \bar x_i$ are vertices of $\Gamma$, and $x_i \in F(A)$ labels a path in $\Gamma$. If $i \ge 1$, then $Hg_{i-1}\bar x_{i-1}$ is also a vertex of $\Gamma$, and there is a $p_i$-labeled path $\pi_i$ in $\Schreier(G,H)$ from $Hg_{i-1}\bar x_{i-1}$ to $Hg_i$. Let $P\in \calP$ be such that $p_i \in P$. If $P \cap H^{g_{i-1}\bar x_{i-1}}$ is infinite, then $\pi_i$ sits inside the image of $\phi_{P,Hg_{i-1}\bar x_{i-1}}$, and hence inside $\Gamma$. If instead $P \cap H^{g_{i-1}\bar x_{i-1}}$ is finite, we note that $p_i \in L_P$ by definition of $L$, and it follows that $\pi_i$ sits inside $\Gamma$ as well.

Therefore $w$ labels a loop of $\Gamma$ at vertex $H$. This shows that $\Gamma$ is Stallings-like for $H$ with respect to $L$, which implies that $H$ is $L$-quasi-convex.
\eopo

This yields the following corollary.

\begin{cor}\label{cor: peripherally finite index}
Let $(G,\calP)$ be a relatively hyperbolic group, such that $\calP$ consists of geodesically bi-automatic groups. 
There are partial algorithms which
\begin{itemize}
\item when they halt on input a tuple $(h_i)_{1\le i\le s}$ of reduced words, output the Stallings graph of the subgroup $H$ generated by the $h_i$ relative to the language $L$ given in Theorem~\ref{thm: peripherally finite index}; compute a constant of relative quasi-convexity for $H$; decide whether $H$ is contained in some peripheral group; decide the membership and the finiteness problems for $H$;

\item halt on a set of inputs which includes the tuples for which $H$ is relatively quasi-convex and has peripherally finite index.
\end{itemize}

Similarly, there are partial algorithms which, when they halt on input two tuples of reduced words generating subgroups $H$ and $K$ of $G$, decide whether $H\le K$ or $H = K$, and compute the intersection $H\cap K$; and which halt at least on input the tuples for which $H$ and $K$ are relatively quasi-convex  and have peripherally finite index.
\end{cor}

%

\preuve
The announced algorithms start with computing $L$ and an automatic structure for $G$ as in Theorem~\ref{thm: peripherally finite index}. We then use the partial algorithms from Sections~\ref{sec: membership problem}, \ref{sec: finiteness problem} and~\ref{sec: intersection}.

These algorithms halt exactly when the subgroups under consideration are $L$-quasi-convex, and this includes the case where they are relatively quasi-convex and have peripherally finite index by Theorem~\ref{thm: peripherally finite index}. 

This takes care of most of the problems mentioned in the statement. To complete the proof, we need to consider the question of computing a constant of relative quasi-convexity, and of inclusion in a peripheral subgroup.

First we consider the computable constants $\lambda$, $\epsilon$ and $\alpha = a(\lambda,\epsilon)$ in Propositions~\ref{prop: Antolin-Ciobanu} and~\ref{prop: improved BCP}, and a constant $\beta$ of $L$-quasi-convexity for $H$, which is also computable since $H$ is $L$-quasi-convex. The proof of Theorem~\ref{thm: peripherally finite index} shows that $\alpha+\beta$ is a constant of relative quasi-convexity for $H$.

Finally we note that if $P$ is a peripheral subgroup, then the $L$-representatives of the elements of $P$ are written on the alphabet $X \cap P$ by Proposition~\ref{prop: Antolin-Ciobanu}. Therefore $H \le P$ if and only if the edges of the Stallings graph of $H$, relative to $L$, are labeled only by letters in $X \cap P$.
\eop

\subsection{Intersection of conjugates of relatively quasi-convex subgroups}

In this section, we examine intersections of conjugates of relatively quasi-convex subgroups, in the spirit of the discussion in Sections~\ref{sec: BICI} and~\ref{sec: conjugacy in hyperbolic}.

Let $(G,\calP)$ be a relatively hyperbolic group.
An element $f\in G$ is said to be \emph{elliptic} if it has finite order, \emph{parabolic} if it has infinite order and it conjugates into a peripheral subgroup, and \emph{hyperbolic} otherwise. A subgroup $H \le G$ is \emph{elliptic} if it is finite, \emph{parabolic} if it is infinite and conjugates into a peripheral subgroup and \emph{hyperbolic} otherwise\footnote{Hyperbolic elements and subgroups are called \emph{loxodromic} in \cite[Sec. 8]{2009:HruskaWise}.}. A hyperbolic subgroup contains a hyperbolic element.

We start with a general result of Mart\'\i nez-Pedroza \cite[Prop. 1.5]{2009:Martinez-Pedroza} (see also Hruska \cite[Thm 9.1]{2010:Hruska}).

\begin{prop}\label{prop: maximal parabolic alt}
Let $(G,\calP)$ be a relatively hyperbolic group, let $P\in \calP$ and let $H$ be a relatively quasi-convex subgroup with constant of relative quasi-convexity $k$. Then every intersection of the form $H \cap P^g$ ($P\in \calP$, $g\in G$) which is infinite, is conjugated in $H$ to a subgroup of the form $H \cap P^x$ with $|x|_A\le k$.
\end{prop}

This translates immediately to the following characterization.

\begin{cor}\label{cor: charaterizing parabolicity alt}
Let $(G,\calP)$ be a relatively hyperbolic group and let $H$ be a relatively quasi-convex subgroup of $G$ with constant of relative quasi-convexity $k$. Then $H$ is parabolic if and only if $H^x \le P$ for some $P\in\calP$ and some $x\in G$ such that $|x|_A\le k$.
\end{cor}

We now consider intersections of conjugates of relatively quasi-convex subgroups, that form hyperbolic subgroups of $G$. Following Hruska and Wise \cite[Definitions 8.1 and 9.1]{2009:HruskaWise}, we define a relative version of the notion of height of a subgroup $H$: the \emph{relative height} $\hauteur_{G,\calP}(H)$ is the maximum number $n$ for which there exist distinct cosets $Hg_1,\ldots, Hg_n$ such that $\bigcap_{i=1}^nH^{g_i}$ is hyperbolic. Also, we say that $H$ is \emph{relatively malnormal} if, for every $g\not\in H$, the intersection $H \cap H^g$ is either elliptic or parabolic.

If $\nu\colon \N\to \N$ is a non-decreasing function, we say that \emph{$(G,\calP)$ has Property $\rBP_\nu$} if, whenever $H$ and $K$ are relatively quasi-convex subgroups of $G$ with constant of relative quasi-convexity $k$, if $K, g_1H, \ldots, g_n$ are pairwise distinct and if $K \cap \bigcap_i H^{g_i}$ is hyperbolic, then there exists an element $z \in G$ such that the ball of center $z$ and radius $\nu(k)$ meets $K$ and each of the $g_iH$.

This notion leads to the following statement, a relatively hyperbolic analogue of Proposition~\ref{prop: BP to double cosets}, with exactly the same proof.

\begin{prop}\label{prop: rBP to double cosets}
Let $\nu$ be a non-decreasing function and let $(G,\calP)$ be a relatively hyperbolic group with Property $\rBP_\nu$. Let $H$ and $K$ be relatively quasi-convex subgroups of $G$ with constant of relative quasi-convexity $k$. Then the following holds.
\begin{itemize}
\item[(1)] Every double coset $HgK$ such that $K \cap H^g$ is hyperbolic, has a representative of length at most $2\nu(k)$.

\item[(2)] There exists a finite family $\calJ$ of hyperbolic intersections $K \cap H^g$ such that any hyperbolic intersection $K \cap H^x$ is conjugated in $K$ to an element of $\calJ$.
\end{itemize}
\end{prop}

Proposition~\ref{prop: packing hyperbolic alt} below contains a relatively hyperbolic version of Proposition~\ref{prop: hyperbolic BP_nu}, which is a small variant of the work of Hruska and Wise characterizing packing in relatively hyperbolic groups. These authors use it to show that relatively quasi-convex subgroups have finite relative height \cite[Cor. 8.6]{2009:HruskaWise}. We use it in Corollary~\ref{cor: relative hyperbolic computability results} to prove in particular that this relative height is computable under suitable hypotheses.

\begin{prop}\label{prop: packing hyperbolic alt}
Let $(G,\calP)$ be an $A$-generated relatively hyperbolic group. There exists a constant $\ell$ depending on $G$, $A$ and $\calP$, such that $G$ satisfies Property $\rBP_\nu$, with $\nu(k) = k + \ell$.

If the word problem is solvable in the peripheral subgroups, then the function $\nu$ (that is: the constant $\ell$) is effectively computable.
\end{prop}

\preuve
Except for the computability of $\ell$, the proof follows exactly the same steps as the proof of Proposition~\ref{prop: hyperbolic BP_nu}, where the basic properties of hyperbolic groups are substituted with the following analogous properties of relatively hyperbolic groups, established by Osin, and by Hruska and Wise.

\begin{fact}[Corollary 4.20 in \cite{2006:Osin}]\label{fact: Osin420 alt}
If $f\in G$ is hyperbolic, then the map $n \mapsto f^n$ ($n\in\Z$) is a quasi-geodesic line in $\Cayley(G,A\cup B)$. It follows that the set $\{f^n \mid n\in\Z\}$ has exactly two limit points in $\partial\Cayley(G,A\cup B)$, denoted by $f^{-\infty}$ and $f^\infty$.
\end{fact}

\begin{fact}[Lemma 8.3 in \cite{2009:HruskaWise}]\label{fact: HW83 alt}
If $f\in H$ is hyperbolic, there exists a bi-infinite relative geodesic path $c$ whose endpoints in $\partial\Cayley(G,A\cup B)$ are $f^{-\infty}$ and $f^\infty$, and whose every vertex is at $A$-distance at most $k$ from $H$.
\end{fact}

\begin{fact}[Lemma 8.2 in \cite{2009:HruskaWise}]\label{fact: HW82 alt}
There exists a constant $\ell$, dependent on $G$, $A$ and $\calP$ only, such that if $c$ and $c'$ are relative geodesic lines with the same endpoints at infinity, then for every vertex $v$ of $c$, there exists a vertex $v'$ of $c'$ such that $\dist_A(v,v') \le \ell$.
\end{fact}

Turning to the effective computability of $\ell$, we note that the proof of  \cite[Lemma 8.2]{2009:HruskaWise} shows how to compute $\ell$ given a certain constant $\nu$ from \cite[Theorem 3.26]{2010:Hruska}. Then \cite[Theorem 3.26]{2006:Osin} shows that $\nu$ can be computed from the isoperimetric function of $\Cayley(G,A\cup B)$. And this function is computable by a result of Dahmani \cite[Thm 0.1]{2008:Dahmani}, under the hypothesis that we can solve the word problem in the peripheral subgroups.
\eop

We now derive decidability and computability results, analogous to Corollary~\ref{cor: good properties}.

\begin{cor}\label{cor: relative hyperbolic computability results}
Let $(G,\calP)$ be a relatively hyperbolic group such that $\calP$ consists of geodesically bi-automatic groups. There is a partial algorithm which, on input two tuples of reduced words generating subgroups $H$ and $K$ of $G$, respectively, 

\begin{itemize}
\item halts on a set of instances which includes the tuples such that $H$ and $K$ are relatively quasi-convex with peripherally finite index;

\item when it halts,
\begin{itemize}
\item[(1)] decides whether $H$ is elliptic, parabolic or hyperbolic;

\item[(2)] computes a family $\calJ$ as in Proposition~\ref{prop: rBP to double cosets} (that is: lists the subgroups in $\calJ$ and produces a finite set of generators for each of them);

\item[(3)] if $K$ is hyperbolic, decides whether $K$ is contained in a conjugate of $H$, or whether it is conjugated to $H$;

\item[(4)] computes $\hauteur_{G,\calP}(H)$;

\item[(5)] decides whether $H$ is relatively malnormal.
\end{itemize}
\end{itemize}
\end{cor}

\preuve
We start with applying the partial algorithms in Corollary~\ref{cor: peripherally finite index}. This allows in particular the computation of a constant $k$ of relative quasi-convexity for $H$ and $K$. 

To prove (1): We already saw that one can decide whether $H$ is finite, that is, whether $H$ is elliptic (Corollary~\ref{cor: peripherally finite index}). If it is not the case, then $H$ is parabolic if there exists $x\in G$ and $P\in \calP$ such that $H^x \le P$, and by Proposition~\ref{prop: maximal parabolic alt}, $x$ can be chosen such that $|x|_A\le k$.

For every word $x$ with $|x|_A\le k$ and every $P\in\calP$, we can decide whether $H^x \le P$ by Corollary~\ref{cor: peripherally finite index} again. Therefore we can decide whether a non-elliptic subgroup $H$ is parabolic. Finally, $H$ is hyperbolic if and only if it is neither elliptic nor parabolic.

The rest of the proofs is as in Proposition~\ref{prop: good properties}, using Corollary~\ref{cor: peripherally finite index} and Propositions~\ref{prop: rBP to double cosets} and~\ref{prop: packing hyperbolic alt}. Note that we can apply the latter since bi-automatic groups have solvable word problems.
\eop

We now turn to intersections of conjugates of relatively quasi-convex subgroups, that form parabolic subgroups of $G$, in order to complete Proposition~\ref{prop: rBP to double cosets} and Corollary~\ref{cor: relative hyperbolic computability results} (2) and (3). To do so, we introduce an additional assumption on $G$.

It is well-known, and elementary, that a non-cyclic free abelian group is not hyperbolic. As a result, if $P$ is a non-cyclic free abelian subgroup of $G$, then $P$ conjugates into some element of $\calP$. The group $G$ is said to be \emph{toral} if $G$ is torsion-free and the elements of $\calP$ are free non-cyclic abelian subgroups of $G$. Note that free abelian groups are geodesically bi-automatic.
We use the following result, due to Groves \cite[Cor. 5.7]{2005:Groves}, \cite[Lemma 6.9]{2009:Groves}.

\begin{prop}\label{prop: CSA alt}
If $(G,\calP)$ is a toral relatively hyperbolic group, then its maximal non-cyclic abelian subgroups are malnormal.
\end{prop}

With the additional hypothesis of toral relative hyperbolicity, we prove the following result.

\begin{prop}\label{prop: parabolic subgroup case alt}
Let $(G,\calP)$ be a toral relatively hyperbolic group and let $H, K$ be relatively quasi-convex subgroups. There exists a finite family $\calJ$ of subgroups of $K$ such that, for every $g\in G$ such that $K\cap H^g$ is parabolic, $K\cap H^g$ is conjugated in $K$ to an element of $\calJ$.

If $k$ is a constant of relative quasi-convexity for $H$ and $K$, then $\calJ$ can be taken to be the family of parabolic subgroups of the form $K\cap H^g$ with $|g|_A \le 2k$.
\end{prop}

\preuve
Let $k$ be a constant of relative quasi-convexity for $H$ and $K$.  Let $g\in G$ be such that $K \cap H^g$ is parabolic, that is, $K \cap H^g$ is infinite and contained in some $P^z$ ($P\in \calP$ and $z\in G$). Then in particular $K \cap P^z$ is infinite and by Proposition~\ref{prop: maximal parabolic alt}, there exist $x\in G$ and $x'\in K$ such that $|x|_A \le k$ and $K \cap P^z = (K \cap P^x)^{x'} = K \cap P^{xx'}$. It follows that $P^z \cap P^{xx'}$ is infinite and by Proposition~\ref{prop: CSA alt}, we have $P^z = P^{xx'}$.

The intersection $H^g \cap P^z = H^g \cap P^{xx'}$ is also infinite, and so is its conjugate $H \cap P^{xx'g\inv}$. By Proposition~\ref{prop: maximal parabolic alt} and Proposition~\ref{prop: CSA alt} again, there exist $y\in G$ and $y'\in H$ such that $|y|_A \le k$ and $P^{xx'g\inv} = P^{yy'}$. It follows that $yy'g{x'}\inv x\inv \in P$, that is: $g = {y'}\inv y\inv pxx'$ for some $p \in P$.

Now we have $K \cap H^g \subseteq P^{xx'}$, that is, $K \cap H^{y\inv pxx'} \subseteq P^{xx'}$. Conjugating by $(xx')\inv$, we find that $K^{x\inv} \cap H^{y\inv p} \subseteq P$. Let $s \in K^{x\inv} \cap H^{y\inv p}$. Then $s \in P$, which is an abelian group, so $s = s^{p\inv}$. It follows that $K^{x\inv} \cap H^{y\inv p} = K^{x\inv} \cap H^y \cap P$, and conjugating back by $xx'$, we get
\begin{align*}
K \cap H^g &= K \cap H^{y\inv xx'} \cap P^{xx'} \\
&= K^{x'} \cap H^{y\inv xx'} \cap P^{xx'} \\
&= (K \cap P^x \cap H^{y\inv x})^{x'}.
\end{align*}
So we can take the set $\calJ$ to be the set of all intersections of the form $K \cap P^x \cap H^{y\inv x}$ which are infinite, with $|x|_A, |y|_A \le k$.
\eop

\begin{cor}\label{cor: parabolic subgroup case alt}
Let $(G,\calP)$ be a toral relatively hyperbolic group. There is a partial algorithm which, given tuples generating subgroups $H$ and $K$ of $G$,
\begin{itemize}
\item halts on a set of instances which includes the tuples such that $H$ and $K$ are relatively quasi-convex with peripherally finite index;

\item when it halts, computes a family $\calJ$ as in Proposition~\ref{prop: parabolic subgroup case alt}.
\end{itemize}
\end{cor}

\preuve
The partial algorithms in Corollary~\ref{cor: peripherally finite index} allow us to compute a constant of relative quasi-convexity $k$ for $H$ and $K$. Then $\calJ$ can be taken to be the parabolic subgroups of the form $K\cap H^g$ where $|g|_A \le 2k$ by Proposition~\ref{prop: parabolic subgroup case alt}. Each intersection $K\cap H^g$, $|g|_A \le 2k$, can be computed and tested for parabolicity by Corollary~\ref{cor: peripherally finite index}.
\eop

Proposition~\ref{prop: rBP to double cosets} and Corollary~\ref{cor: parabolic subgroup case alt} yield the following corollaries.

\begin{cor}\label{cor: relatively hyperbolic alt}
Let $(G,\calP)$ be a toral relatively hyperbolic group.

For all relatively quasi-convex subgroups $H$ and $K$ of $G$, there exists a finite family $\calJ$ of subgroups of $K$ such that every intersection of the form $K\cap H^g$ is either finite, or conjugated in $K$ to an element of $\calJ$.

There is a partial algorithm which, on input two finite tuples generating subgroups $H$ and $K$ of $G$,
\begin{itemize}
\item halts on a set of instances which includes the tuples such that $H$ and $K$ are relatively quasi-convex with peripherally finite index;

\item when it halts, computes such a family $\calJ$.
\end{itemize}
\end{cor}

\begin{cor}\label{cor: deciding stuff alt}
Let $G$ be a toral relatively hyperbolic group. There is a partial algorithm which, on input two finite tuples generating subgroups $H$ and $K$ of $G$,
\begin{itemize}
\item halts on a set of instances which includes the tuples such that $H$ and $K$ are relatively quasi-convex with peripherally finite index;

\item when it halts, decides whether $K$ is infinite, and in that case decides whether $K$ is contained in (resp. equal to) a conjugate of $H$.
\end{itemize}
\end{cor}

\preuve
Note that if $K$ is infinite, then $K$ is contained in a conjugate of $H$ if $K = K\cap H^g$ for some $g$, if and only if $K$ belongs to the family $\calJ$ described in Corollary~\ref{cor: relatively hyperbolic alt}.

The partial algorithm in Corollary~\ref{cor: peripherally finite index}, with the announced halting properties, allows us to decide whether $K$ is infinite. If it is the case, Corollary~\ref{cor: relatively hyperbolic alt} provides a partial algorithm, again with the appropriate halting properties, which outputs (when it halts) the family $\calJ$. Finally, Corollary~\ref{cor: peripherally finite index} shows how to decide whether $K = J$ for every $J\in \calJ$. 
\eop

Corollaries~\ref{cor: relatively hyperbolic alt} and~\ref{cor: deciding stuff alt} can be compared to the following result due to Kharlampovich \emph{et al.} \cite{2004:KharlampovichMyasnikovRemeslennikov} about limit groups, which does not require the subgroups to have peripherally finite index. Notice that limit groups are particular toral relatively hyperbolic, which are locally relatively quasi-convex.


\begin{theorem} \label{int}
Let $G$ be a limit group and let $H$ and $K$ finitely generated subgroups of $G$ given by finite generating sets.  Then one can effectively compute a finite family $\calJ$ of non-trivial finitely generated subgroups of $G$ (each given by finite generating sets), such that
\begin{itemize}
\item every element of $\calJ$ is of one of the following types
$$H^{g_1} \cap K,\qquad H^{g_1} \cap C_K(g_2)$$
where $g_1 \in G \smallsetminus H$, $g_2 \in K$, $C_K(g_2)$ is the centralizer of $g_2$ in $K$, and $g_1$ and $g_2$ can be computed;
\item for any non-trivial intersection $H^g \cap K,\ g \in G \smallsetminus H$ is equal to a subgroup of the form $J^x$, where $J\in \calJ$ and $x\in K$ can be computed.
\end{itemize}
In  particular, one can decide whether $H$ and $K$ are conjugated.
\end{theorem}

\paragraph{Acknowledgements}
The authors thank the anonymous referee whose questions helped us make the statements of some of our results more precise and rigorous.

{\small

%
%
%
}

\end{document}